\documentclass{eptcs-modified}

\usepackage{amsmath}
\usepackage{amssymb}
\usepackage{url}
\usepackage{amsthm}
\usepackage{graphicx}
\usepackage{amscd}
\usepackage[usenames,dvipsnames,svgnames,table]{xcolor}

\begin{document}

\theoremstyle{definition}
\newtheorem{theorem}{Theorem}
\newtheorem{definition}[theorem]{Definition}
\newtheorem{problem}[theorem]{Problem}
\newtheorem{assumption}[theorem]{Assumption}
\newtheorem{corollary}[theorem]{Corollary}
\newtheorem{proposition}[theorem]{Proposition}
\newtheorem{example}[theorem]{Example}
\newtheorem{lemma}[theorem]{Lemma}
\newtheorem{observation}[theorem]{Observation}
\newtheorem{fact}[theorem]{Fact}
\newtheorem{question}[theorem]{Open Question}
\newtheorem{conjecture}[theorem]{Conjecture}
\theoremstyle{theorem}
\newtheorem{addendum}[theorem]{Addendum}
\newtheorem{remark}[theorem]{Remark}
\newcommand{\uint}{{[0, 1]}}
\newcommand{\Cantor}{{\{0,1\}^\mathbb{N}}}
\newcommand{\name}[1]{\textsc{#1}}
\newcommand{\id}{\textrm{id}}
\newcommand{\dom}{\operatorname{dom}}
\newcommand{\Baire}{{\mathbb{N}^\mathbb{N}}}
\newcommand{\hide}[1]{}
\newcommand{\mto}{\rightrightarrows}
\newcommand{\Sierp}{Sierpi\'nski }
\newcommand{\C}{\textrm{C}}
\newcommand{\VC}{\textrm{VC}}
\newcommand{\UC}{\textrm{UC}}
\newcommand{\lpo}{\textrm{LPO}}
\newcommand{\llpo}{\textrm{LLPO}}
\newcommand{\leqW}{\leq_{\textrm{W}}}
\newcommand{\leqsW}{\leq_{\textrm{sW}}}
\newcommand{\leW}{<_{\textrm{W}}}
\newcommand{\equivW}{\equiv_{\textrm{W}}}
\newcommand{\equivT}{\equiv_{\textrm{T}}}
\newcommand{\geqW}{\geq_{\textrm{W}}}
\newcommand{\pipeW}{|_{\textrm{W}}}
\newcommand{\nleqW}{\nleq_\textrm{W}}
\newcommand{\argmin}{\operatorname{argmin}}
\newcommand{\Sort}{\operatorname{Sort}}
\newcommand{\range}{\operatorname{ran}}
\newcommand{\embeds}{\hookrightarrow}
\newcommand{\NN}{{\mathbb{N}}}
\newcommand{\PU}{{\mathbf{K}}}

\newcommand{\XX}{\mathbf{X}}
\newcommand{\YY}{\mathbf{Y}}
\newcommand{\Smin}{{\mathbf{S}_{\mathrm{min}}}}
\newcommand{\Tmin}{{\mathbf{T}_{\mathrm{min}}}}
\newcommand{\leqWcont}{\leq_{\mathrm{W}}^{\mathrm{t}}}
\newcommand{\leqsWcont}{\leq_{\mathrm{sW}}^{\mathrm{t}}}
\newcommand{\leWcont}{<_{\mathrm{W}}^{\mathrm{t}}}
\newcommand{\equivWcont}{\equiv_{\mathrm{W}}^{\mathrm{t}}}

\newcommand\tboldsymbol[1]{%
\protect\raisebox{0pt}[0pt][0pt]{%
$\underset{\widetilde{}}{\boldsymbol{#1}}$}\mbox{\hskip 1pt}}
\newcommand{\bolds}{\tboldsymbol{\Sigma}}
\newcommand{\boldp}{\tboldsymbol{\Pi}}
\newcommand{\boldd}{\tboldsymbol{\Delta}}
\newcommand{\boldg}{\tboldsymbol{\Gamma}}

\newcommand{\todo}[2]{{\bf TODO } by {\bf #1}: \emph{#2} {\bf end}}
\newcommand{\comment}[2]{{\bf #1}: \emph{#2} {\bf end}}

\title{Overt choice}

\author{
Matthew de Brecht
\institute{Graduate School of Human and Environmental Studies\\ Kyoto University, Japan\\}
\email{matthew@i.h.kyoto-u.ac.jp}
\and
Arno Pauly
\institute{Swansea University\\Swansea, UK\\ \& \\ University of Birmingham\\ Birmingham, UK\\}
\email{Arno.M.Pauly@gmail.com}
\and
Matthias Schr\"oder
\institute{TU
 Darmstadt, Darmstadt, Germany}\email{Matthias.Schroeder@cca-net.de}
}

\def\titlerunning{Overt choice}
\def\authorrunning{M.~de Brecht, A. Pauly \& M.~Schr\"oder}
\maketitle

\begin{abstract}
We introduce and study the notion of \emph{overt choice} for countably-based spaces and for CoPolish spaces. Overt choice is the task of producing a point in a closed set specified by what open sets intersect it. We show that the question of whether overt choice is continuous for a given space is related to topological completeness notions such as the Choquet-property; and to whether variants of Michael's selection theorem hold for that space. For spaces where overt choice is discontinuous it is interesting to explore the resulting Weihrauch degrees, which in turn are related to whether or not the space is Fr\'echet-Urysohn.
\end{abstract}

{\renewcommand*{\thefootnote}{*} \footnotetext{\noindent\begin{minipage}{0.1\textwidth}\includegraphics[width=\textwidth]{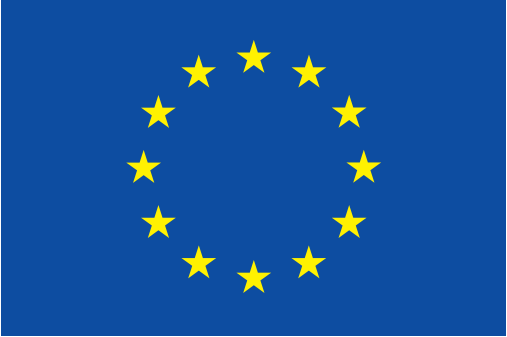}\end{minipage} \begin{minipage}{0.8\textwidth} This project has received funding from the European Union�s Horizon 2020 research and innovation programme under the Marie Sklodowska-Curie grant agreement No 731143, \emph{Computing with Infinite Data}.\end{minipage}

The first author was supported by JSPS Core-to-Core Program, A. Advanced Research Networks and by JSPS KAKENHI Grant Number 18K11166.}}

\section{Introduction}
Let us assume that we have the ability to recognize, given an open predicate, there exists a solution satisfying that predicate. Under which conditions does this suffice to actually obtain a solution? This idea is formalized in the notion of \emph{overt choice}. In this paper, we investigate overt choice under the assumption that the set of solutions is a topologically closed set, although we will often omit the word ``closed'' for simplicity.

On the one hand, studying overt choice is a contribution to (computable) topology. Overt choice for a given space being computable/continuous is a completeness notion, which in the metric case coincides with being Polish. Understanding in more generality for what spaces overt choice is continuous will aid us in extending results and constructions from Polish spaces to more general classes of spaces. On the other hand, the degrees of non-computability of choice principles have turned out to be an extremely useful scaffolding structure in the Weihrauch lattice. Studying the Weihrauch degrees of overt choice in spaces where this is not computable reveals more about hitherto unexplored regions of the Weihrauch lattice.

Overtness is the often overlooked dual notion to compactness. A subset of a space is overt, if the set of open subsets intersecting it is itself an open subset of the corresponding hyperspace. Equivalently, if existential quantification over the set preserves open predicates. Since in classical topology, arbitrary unions of open sets are open, overtness becomes trivial. In constructive or synthetic topology, however, it is a core concept. Even classically, though, we can make sense of the space $\mathcal{V}(\mathbf{X})$ of (topologically closed) overt subsets of a given space $\mathbf{X}$. The space $\mathcal{V}(\mathbf{X})$ is isomorphic to the hyperspace of closed sets with the positive information topology (equivalently, the sequentialization of the lower Vietoris or lower Fell topology). Overt choice for $\mathbf{X}$, which we denote $\VC_\mathbf{X}$, is just the task of producing an element of a given non-empty set $A \in \mathcal{V}(\mathbf{X})$.

Since the lower-semicontinuous closed-valued functions into $\mathbf{X}$ are equivalent to the continuous functions into $\mathcal{V}(\mathbf{X})$, the continuity of $\VC_\mathbf{X}$ gives rise to variants of Michael's selection theorem \cite{michael}. We can thus view the question of what spaces make overt choice continuous as asking about for which spaces Michael's selection theorem holds.

\paragraph*{Our contributions}
We generalize the known result that overt choice is computable for computable Polish spaces to computable quasi-Polish spaces (Theorem \ref{theo:overtchoicequasipolish}). Since the latter notion is not yet fully established, we first investigate a few candidate definitions for effectivizing the notion of a quasi-Polish space, and show that the candidate definitions fall into two equivalence classes, which we then dub \emph{precomputably quasi-Polish}(Definition \ref{def:precomputableQP}) and \emph{computably quasi-Polish} (Definition \ref{def:computableqp}).

As a partial converse, we show that for countably-based $T_1$-spaces the continuity of overt choice is equivalent to being quasi-Polish (Corollary \ref{corr:overtchoiceqpcharac}). In Section \ref{sec:other} we then explore overt choice for several canonic examples of countably-based yet not quasi-Polish spaces, and study the Weihrauch degrees of overt choice for these.

Besides countably-based spaces, we also investigate CoPolish spaces (Section \ref{sec:copolish}). We see that overt choice is continuous for a CoPolish space iff that space is actually countably-based. Moreover, the topological Weihrauch degree of overt choice on a CoPolish space is always comparable with $\lpo$, and whether it is above or strictly below $\lpo$ tells us whether the space has the Fr\'echet-Urysohn property. These results are summarized in Corollary \ref{corr:copolishsummary}.

\section{Background on represented spaces and Weihrauch degrees}

\subsection{Represented spaces and synthetic topology}
The formal setting for our investigation will be the category of represented spaces \cite{pauly-synthetic}, which is commonly used in computable analysis. It constitutes a model for synthetic topology in the sense of Escard\'o \cite{escardo}. We will contend ourselves with giving a very brief account of the essential notions for our purposes, and refer to \cite{pauly-synthetic} for more details and context. In the area of computable analysis, most of the following were first obtained in \cite{schroder5}.

\begin{definition}
A \emph{represented space} is a pair $(X,\delta)$ of a set $X$ and a partial surjection $\delta : \subseteq \Baire \to X$. We commonly write $\mathbf{X}$, $\mathbf{Y}$, etc, for represented spaces $(X,\delta_X)$, $(Y,\delta_Y)$. A (multivalued) function between represented spaces is a (multivalued) function between the underlying sets. A partial function $F : \subseteq \Baire \to \Baire$ is a \emph{realizer} of $f : \subseteq \mathbf{X} \mto \mathbf{Y}$ (written $F \vdash f$) if $\delta_Y(F(p)) \in f(\delta_X(p))$ for all $p \in \dom(f \circ \delta_X)$. We call $f : \subseteq \mathbf{X} \mto \mathbf{Y}$ \emph{computable} respectively \emph{continuous} if it has a computable respectively continuous realizer.
\end{definition}

By the grace of the UTM-theorem, the category of represented spaces and continuous functions is cartesian-closed, i.e.~we have a function space construction -- that even makes all the expected operations computable. We denote the space of continuous functions from $\mathbf{X}$ to $\mathbf{Y}$ by $\mathcal{C}(\mathbf{X},\mathbf{Y})$. A special represented space of significant relevance is Sierpi\'nski space $\Sigma$, having the two elements $\top$ and $\bot$ and represented via $\delta_\Sigma : \Baire \to \{\top,\bot\}$ where $\delta_\Sigma^{-1}(\{\bot\}) = \{0^\omega\}$.

We obtain a space $\mathcal{O}(\mathbf{X})$ of subsets of a given space $\mathbf{X}$ by identifying $U \subseteq \mathbf{X}$ with its characteristic function $\chi_U : \mathbf{X} \to \Sigma$. The elements of $\mathcal{O}(\mathbf{X})$ are called \emph{open} sets, which is justified in particular by noting that $\mathcal{O}(\mathbf{X})$ is the final topology induced along $\delta_\mathbf{X}$ by the subspace topology on $\dom(\delta_\mathbf{X})$. The space $\mathcal{A}(\mathbf{X})$ of closed subsets is obtained by identifying a set $A \in \mathcal{A}(\mathbf{X})$ with its complement $(X \setminus A) \in \mathcal{O}(\mathbf{X})$.

The space $\mathcal{O}(\mathbb{N})$ has a particularly nice characterization: Its elements are all subsets of $\mathbb{N}$, and they are represented as enumerations. The topology on $\mathcal{O}(\mathbb{N})$ is thus the Scott topology. We then define the notion of an effective countable basis:

\begin{definition}
$\mathbf{X}$ is effectively countably-based if there is some computable function $B : \mathbb{N} \to \mathcal{O}(\mathbf{X})$ such that the computable function $U \mapsto \bigcup_{n \in U} B(n) : \mathcal{O}(\mathbb{N}) \to \mathcal{O}(\mathbf{X})$ has a computable multi-valued inverse.
\end{definition}



For any subset $A \subseteq \mathbf{X}$, the set  $U_A = \{U \in \mathcal{O}(\mathbf{X}) \mid U \cap A \neq \emptyset\}$ is an open subset of $\mathcal{O}(\mathbf{X})$, hence an element of  $\mathcal{O}(\mathcal{O}(\mathbf{X}))$. The corresponding characteristic function $\chi_{U_A} : \mathcal{O}(\mathbf{X}) \to \Sigma$ is continuous and preserves finite joins between the lattices $\mathcal{O}(\mathbf{X})$ and $\Sigma$ (i.e., $\chi_{U_A}(\emptyset)=\bot$ and $\chi_{U_A}(U\cup V) = \chi_{U_A}(U)\vee \chi_{U_A}(V)$). Conversely, if $\chi : \mathcal{O}(\mathbf{X}) \to \Sigma$ is continuous and preserves finite joins, then by defining $A = X\setminus \bigcup\{U\in \mathcal{O}(\mathbf{X}) \mid \chi(U)=\bot\}$, we see that $\chi$ is the characteristic function of the open subset $\{U \in \mathcal{O}(\mathbf{X}) \mid U \cap A \neq \emptyset\}$ of $\mathcal{O}(\mathbf{X})$.

We define $\mathcal{V}(\mathbf{X})$ to be the subspace of $\mathcal{O}(\mathcal{O}(\mathbf{X}))$ of join preserving functions in the above sense. Since any $A \subseteq \mathbf{X}$ determines an element $\{U \in \mathcal{O}(\mathbf{X}) \mid U \cap A \neq \emptyset\}$ in $\mathcal{V}(\mathbf{X})$ which encodes the information about which open sets intersect $A$, it is convenient to think of $\mathcal{V}(\mathbf{X})$ as the \emph{space of overt subsets} of $\mathbf{X}$. However, this does not characterize subsets of $\mathbf{X}$ uniquely: For sets $A, B \subseteq \mathbf{X}$ we have that $\{U \in \mathcal{O}(\mathbf{X}) \mid U \cap A \neq \emptyset\} = \{U \in \mathcal{O}(\mathbf{X}) \mid U \cap B \neq \emptyset\}$ if and only if $A$ and $B$ have equal closures in $\mathbf{X}$.

To avoid this ambiguity, we adopt the convention that the elements of $\mathcal{V}(\mathbf{X})$ are encoding topologically closed subsets of $\mathbf{X}$. Under this convention, the space $\mathcal{V}(\mathbf{X})$ is isomorphic to the space of closed subsets of $\mathbf{X}$ with the positive information topology (equivalently, the sequentialization of the lower Vietoris or lower Fell topology). We will, however, sometimes simply refer to the elements of  $\mathcal{V}(\mathbf{X})$ as ``overt'' sets, with the implicit understanding that they are topologically closed.

If $\mathbf{X}$ is effectively countably-based via some $(B_n)_{n \in \mathbb{N}}$, we can conceive of $A \in \mathcal{V}(\mathbf{X})$ as being represented via $\{n \in \mathbb{N} \mid B_n \cap A \neq \emptyset\} \in \mathcal{O}(\mathbb{N})$.

\subsection{Weihrauch degrees}
Weihrauch reducibility is a preorder between multivalued functions on represented spaces. It is a many-one reducibility captures the idea of when $f$ is solvable using computable means and a single application of another principle $g$. Inspired by earlier work by Weihrauch \cite{weihrauchb,weihrauchc} it was promoted as a setting for computable metamathematics in \cite{gherardi,brattka2,brattka3}. A recent survey and introduction is found in \cite{pauly-handbook}, to which we refer for further reading.

\begin{definition}
Let $f : \subseteq \mathbf{X} \mto \mathbf{Y}$ and $g : \subseteq \mathbf{U} \mto \mathbf{V}$ be multi-valued functions between represented spaces. We say that $f$ is Weihrauch reducible to $g$ ($f \leqW g$) if there are computable functions $H, K : \subseteq \Baire \to \Baire$ such that whenever $G \vdash g$, then $K(\langle \id, GH\rangle) \vdash f$. If we can even chose $H,K$ such that $KGH \vdash f$ whenever $G \vdash g$, we have a strong Weihrauch reduction ($f \leqsW g$).
\end{definition}

We write $\leqWcont$ respectively $\leqsWcont$ for the relativized versions (equivalently, the versions where \emph{computable} is replaced by \emph{continuous}). With $f \equivW g$ we abbreviate $f \leqW g \wedge g \leqW f$, with $f \leW g$ we abbreviate $f \leqW g \wedge g \nleqW f$, and $f \pipeW g$ stands in for $f \nleqW g \wedge g \nleqW f$.

The equivalence classes for $\leqW$ are the Weihrauch degrees, which form a distributive lattice. The cartesian product of multivalued functions induces an operation $\times$ on the Weihrauch degrees. We also use the closure operator $\widehat{\phantom{f}}$ which is induced by the lifting of $f : \subseteq \mathbf{X} \mto \mathbf{Y}$ to $\widehat{f} : \subseteq \mathbf{X}^\mathbb{N} \mto \mathbf{Y}^\mathbb{N}$. The operation $f \star g$ captures the idea of first making one call to $g$, and then one call to $f$. As such, $f \star g$ is the maximal Weihrauch degree arising as a composition $f' \circ g'$ where $f' \leqW f$ and $g' \leqW g$. A formal construction is found in \cite{paulybrattka4}.

Many computational tasks that have been classified in the Weihrauch lattice turned out to be equivalent to a \emph{closed choice} principle parameterized by some represented space:

\begin{definition}
\label{def:closedchoice}
For represented space $\mathbf{X}$, let its closed choice $\C_\mathbf{X} : \subseteq \mathcal{A}(\mathbf{X}) \mto \mathbf{X}$ be defined by $A \in \dom(\C_\mathbf{X})$ iff $A \neq \emptyset$ and $x \in \C_\mathbf{X}(A)$ iff $x \in A$.
\end{definition}

The topological Weihrauch degree is $\C_\mathbf{X}$ reflects topological properties of $\mathbf{X}$. For example, for any uncountable compact metric space $\mathbf{X}$ we find that $\C_\mathbf{X} \equivWcont \C_\Cantor$. This is discussed further in \cite{paulybrattka}. The principle $\C_\mathbb{N}$ has the more intuitive characterization of finding a natural number not occurring in an enumeration (that does not exhaust all natural numbers). Both $\C_\Cantor$ and $\C_\Baire$ are about finding infinite paths through ill-founded trees. For $\C_\Cantor$, the tree is binary, whereas for $\C_\Baire$ the tree can be countably-branching.

\section{Fundamentals on overt choice}
As mentioned above, \emph{overt choice} is the task of finding a point in a given overt set. A priori, this is an ill-specified task, as overt sets do not uniquely determine an actual set of points. Our convention that elements of $\mathcal{V}(\mathbf{X})$ are topologically closed does ensure the well-definedness of overt choice. Consequently, it might be more accurate to speak of \emph{closed overt choice}. To keep notation simple, we omit the reminder of our convention in the following.

\begin{definition}
\label{def:overtchoice}
For represented space $\mathbf{X}$, let its overt choice $\VC_\mathbf{X} : \subseteq \mathcal{V}(\mathbf{X}) \mto \mathbf{X}$ be defined by $A \in \dom(\VC_\mathbf{X})$ iff $A \neq \emptyset$ and $x \in \VC_\mathbf{X}(A)$ iff $x \in A$.
\end{definition}

Note the similarity between the definitions of closed choice (Definition \ref{def:closedchoice}) and overt choice (Definition \ref{def:overtchoice}). Given the importance of the former in the study of Weihrauch degrees, this is an argument in favour of exploring the latter notion, too.

We shall observe some basic properties of how overt choice for various spaces is related, similar to the investigation for closed choice in \cite{paulybrattka}.

\begin{proposition}
Let $s : \mathbf{X} \to \mathbf{Y}$ be an effectively open computable surjection. Then $\VC_\mathbf{Y} \leqW \VC_\mathbf{X}$.
\begin{proof}
As $s$ is effectively open, we can compute $s^{-1}(A) \in \mathcal{V}(\mathbf{X})$ from $A \in \mathcal{V}(\mathbf{Y})$. Then $\VC_\mathbf{X}$ can be used to obtain some $x \in s^{-1}(A)$. Computability of $s$ then lets us compute $s(x) \in A$.
\end{proof}
\end{proposition}

\begin{corollary}
If $\mathbf{X}$ and $\mathbf{Y}$ are computably isomorphic, then $\VC_\mathbf{X} \equivW \VC_\mathbf{Y}$.
\end{corollary}

\begin{proposition}
Let $\mathbf{X}$ be a computably closed subspace of $\mathbf{Y}$. Then $\VC_\mathbf{X} \leqW \VC_\mathbf{Y}$.
\begin{proof}
Under the given conditions, we can compute $\id : \mathcal{V}(\mathbf{X}) \to \mathcal{V}(\mathbf{Y})$, which yields the claim. To see this, we just need to note that $U \mapsto U \cap X : \mathcal{O}(\mathbf{Y}) \to \mathcal{O}(\mathbf{X})$ is computable, and that for $A \subseteq \mathbf{X}$ we have that $U \cap A \neq \emptyset$ iff $(U \cap X) \cap A \neq \emptyset$.
\end{proof}
\end{proposition}

The requirement of the subspace being computably closed is necessary for the previous proposition to hold. To see this, note that we can adjoin a computable bottom element to an arbitrary represented space, in such a way that the original space is computably open inside the resulting space, and that the only open set containing the bottom element is the entire space. Since every non-empty closed subset of the resulting space contains the bottom element, closed overt choice becomes trivially computable.

\begin{lemma}
The map $\mathbf{\times}\colon \mathcal{V}(\XX) \times
 \mathcal{V}(\YY) \to \mathcal{V}(\XX \times \YY)$
  defined by $(A,B) \mapsto A \times B$ is computable.
 \end{lemma}

 \begin{proof}
  For all $x \in \XX$ and all $W \in \mathcal{O}(\XX \times \YY)$
  the set $V_{x,W}:= \{ y \in \YY \,|\, (x,y) \in W\}$ is open
  and the map $(x,W) \mapsto V_{x,W}$ is computable.
By composition, it follows that \[(W,B) \mapsto U_{W,B} := \{ x \in \XX \,|\, B \cap V_{x,W} \neq \emptyset
 \} : \mathcal{O}(\XX \times \YY) \times \mathcal{V}(\YY) \to \mathcal{O}(\XX)\] is well-defined and computable.
  For all $A  \in \mathcal{V}(\XX)$,
  we have $A \cap U_{W,B} \neq \emptyset \Longleftrightarrow (A \times
 B) \cap W \neq \emptyset$.
  Therefore the function $\mathcal{V}(\XX) \times \mathcal{V}(\YY)\times
 \mathcal{O}(\XX \times \YY) \to \Sigma$
  mapping $(A,B,W)$ to $\top$ iff $W$ intersects $A \times B$ is
 computable.
 \end{proof}

 We conclude:

\begin{corollary}
\label{corr:vcproducts}
$\VC_\XX \times \VC_\YY \leqsW \VC_{\XX \times \YY}$.
\end{corollary}

\section{Overt choice for quasi-Polish spaces}
Quasi-Polish spaces were introduced in \cite{debrecht6} as a suitable setting for descriptive set theory. They generalize both Polish spaces, which form the traditional hunting grounds of descriptive set theory, as well as $\omega$-continuous domains which had been the focus of previous work to extend descriptive set theory (e.g.~\cite{selivanov3,selivanov6}). Essentially, they are complete countably-based spaces.

\subsection{Defining computable quasi-Polish spaces}
An attractive feature of the class of quasi-Polish spaces is the multitude of very different yet equivalent definitions for it, as demonstrated in \cite{debrecht6}. This makes the task of identifying the \emph{correct} definition of a computable quasi-Polish space challenging, however: It does not suffice to effectivize one definition, but one needs to check to what extent the classically equivalent definitions remain equivalent in the computable setting, and in case they are not all equivalent, to choose which one is the most suitable definition. While we do not explore effectivizations of all characterizations of quasi-Polish spaces here, we exhibit two classes of definitions equivalent up to computable isomorphism, and propose those as \emph{precomputable quasi-Polish spaces} and \emph{computable quasi-Polish spaces}.

The task of effectivizing the definition of a quasi-Polish space has already been considered by V.~Selivanov \cite{selivanov8} and by  M.~Korovina and O.~Kudinov \cite{korovina2}. We discuss the relationship between the various proposals in the remark after Definition~\ref{def:computableqp} below.

\begin{definition}
Given a transitive binary relation $\prec$ on $\mathbb{N}$, we say that $I \subseteq \mathbb{N}$ is a rounded ideal for $\prec$, iff the following are satisfied:
\begin{enumerate}
\item $I\not=\emptyset$
\item $y \in I \wedge x \prec y \Rightarrow x \in I$
\item $x, y \in I \Rightarrow \exists z \in I. \ x \prec z \wedge y \prec z$
\end{enumerate}
Let $\mathrm{RI}(\prec) \subseteq \mathcal{O}(\mathbb{N})$ be the space of rounded ideals of $\prec$ equipped with the subspace topology\footnote{Note that if $\prec$ is actually a (reflexive) partial order, then $\mathrm{RI}(\prec)$ is the set of all ideals of $\prec$ in the usual sense. On the other hand, if $\prec$ is anti-reflexive, then a rounded ideal of $\prec$ will not have a maximal element. }.
 \end{definition}

For $(\mathbb{N},\prec)$, let the extendability predicate $E \subseteq \mathbb{N}$ be defined as $n \in E$ iff there exists a rounded ideal $I \ni n$. This is equivalent to the existence of an infinite increasing chain in $(\mathbb{N},\prec)$ containing $n$.

Recall that a representation $\delta$ of a represented space $\mathbf{X}$ is \emph{effectively fiber-overt}, if $x \mapsto \overline{\delta^{-1}(\{x\})} : \mathbf{X} \to \mathcal{V}(\Baire)$ is computable. This notion is studied in \cite{pauly-kihara-arxiv,paulybrattka4}. It is closely related to the representation being effectively open.

\begin{theorem}
\label{theo:precomputableQP}
The following are equivalent for a represented space $\mathbf{X}$:
\begin{enumerate}
\item There exists a c.e.~transitive relation $\mathalpha{\prec} \subseteq \mathbb{N} \times \mathbb{N}$ such that $\mathbf{X} \cong \mathrm{RI}(\prec)$.
\item $\mathbf{X}$ is computably isomorphic to a $\Pi^0_2$-subspace of $\mathcal{O}(\mathbb{N})$.
\item $\mathbf{X}$ admits an effectively fiber-overt computably admissible representation $\delta$ such that $\dom(\delta) \subseteq \Baire$ is $\Pi^0_2$.
\end{enumerate}
\begin{proof}
If $\prec$ is c.e., then the property of being a rounded ideal is $\Pi^0_2$. Hence, $1$ implies $2$. To see that $2.$ implies $3.$, just take the post-restriction of the standard representation of $\mathcal{O}(\mathbb{N})$ to the relevant subspace, and copy via the isomorphism to $\mathbf{X}$. This preserves effective fiber-overtness and computable admissibility. The non-trivial step is the implication from $3.$ to $1.$


Let $\mathcal{P}_{\mathrm{fin}}(\mathbb{N})$ denote the space of finite subsets of $\mathbb{N}$ (given as unordered tuples). We construct a c.e.~transitive relation $\prec$ on $\mathcal{P}_\mathrm{fin}(\mathbb{N}) \times \mathbb{N}$, but the translation to $\mathbb{N}$ is straight-forward.

A computable realizer of fiber-overtness will, given a name $p \in \dom(\delta)$ and some $w \in \mathbb{N}^*$ confirm if there is some name $q$ extending $w$ with $\delta(p) = \delta(q)$, if this is the case. By observing when this realizer provides its confirmations we obtain a computable function $f\colon \mathbb{N}^* \times \mathbb{N} \to \mathcal{P}_{\mathrm{fin}}(\mathbb{N}^*)$ such that
\begin{enumerate}
\item
$w\in f(u,n)$ implies that if $u$ can be extended to some $p\in \dom(\delta)$, then $w$ can be extended to some $q\in\dom(\delta)$ satisfying $\delta(p)=\delta(q)$
\item
If $p,q\in\dom(\delta)$ and $\delta(p)=\delta(q)$, then for any prefix $w$ of $q$ there is a prefix $u$ of $p$ and $n\in\mathbb{N}$ such that $w\in f(u,n)$
\item
$u \in f(u,n)$
\item
$f(u,n)$ is closed under prefixes
\item
$f(u,n) \subseteq f(u',n)$ whenever $u'$ extends $u$
\item
$f(u,m) \subseteq f(u,n)$ whenever $m \leq n$
\end{enumerate}

Since $\dom(\delta)$ is a $\Pi^0_2$-subset of $\Baire$, we can understand it to be given via a computable function $\lambda : \mathbb{N}^* \to \mathbb{N}$ which is order-preserving (prefix-order on $\mathbb{N}^*$, standard order on $\mathbb{N}$) such that $p \in \dom(\delta)$ iff $\{\lambda(p_{\leq n}) \mid n \in \mathbb{N}\}$ is unbounded.

Now we define $\prec$ on $\mathcal{P}_\mathrm{fin}(\mathbb{N}) \times \mathbb{N}$ as $(A, n) \prec (B, m)$ iff the following all hold:
\begin{enumerate}
\item
$B\not=\emptyset$
\item
$n < m$
\item
$n < \lambda(w)$ for each $w \in B$
\item
$A \subseteq f(u, m)$ for each $u\in B$
\item
Each $w\in \bigcup_{u\in A} f(u,n)$ has an extension $w '\in B$
\end{enumerate}
One easily checks that $\prec$ is a c.e.~transitive relation.

For effectively fiber-overt and computably admissible $\delta$, the map $x \mapsto \{w \in \mathbb{N}^* \mid w\Baire \cap \delta^{-1}(x)\not=\emptyset\} : \mathbf{X} \to \mathcal{O}(\mathbb{N}^*)$ is an embedding.  The map $U \mapsto \{A \in \mathcal{P}_{\mathrm{fin}}(\mathbb{N}^*) \mid \forall w \in A. \, w \in U\} : \mathcal{O}(\mathbb{N}^*) \to \mathcal{O}(\mathcal{P}_{\mathrm{fin}}(\mathbb{N}^*))$ is an embedding, too. To conclude our proof, we show that the range of the composition of these two embeddings coincides with the rounded ideals of $\prec$.

Let $U_x := \{w \in \mathbb{N}^* \mid w\Baire \cap \delta^{-1}(x)\not=\emptyset\}$ and $F_x = \{(A, n) \in \mathcal{P}_{\mathrm{fin}}(\mathbb{N}^*)\times\mathbb{N} \mid A \subseteq U_x\}$. First, we shall see that $F_x$ is indeed a rounded ideal for $\prec$. Clearly $F_x$ is non-empty. If $(B,m) \in F_x$ and $(A,n) \prec (B,m)$, then every $u\in A$ has an extension $u' \in B\subseteq U_x$, hence $(A,n) \in F_x$. Now consider any pair $(A,n), (B,m) \in F_x$. If $A$ and $B$ are both empty, we can choose any $w\in U_x$ with $n+m < \lambda(w)$, and get $(\{w\}, n+m+1)$ as a joint $\prec$-upper bound for $A$ and $B$ in $F_x$. If $A$ or $B$ is non-empty, then we construct a joint $\prec$-upper bound $(C,r)$ in $F_x$ by defining $C$ as a finite set of suitable extensions $w'$ to each $w\in \bigcup_{u\in A\cup B} f(u,n+m)$. To see how this can be done, note that if $w\in f(u,n+m)$ for some $u \in A\cup B$, then since there is a name $p$ of $x$ extending $u$ there must exist a name $q$ of $x$ extending $w$. So we can choose any prefix $w'$ of $q$ long enough that $n+m < \lambda(w')$ and $A\cup B \subseteq f(w', r_w)$ for large enough $r_w$. Taking the set $C$ of prefixes $w'$ chosen in this way and letting $r\in\mathbb{N}$ be larger than all the corresponding $r_w$, we get a joint $\prec$-upper bound for $(A,n)$ and $(B,m)$ satisfying $(C,r) \in F_x$. Therefore, $F_x \in \mathrm{RI}(\prec)$.

It remains to argue that any rounded ideal $F$ for $\prec$ is of the form $F_x$. Given a rounded ideal $F$, consider the set $N \subseteq \Baire$ consisting of those $p$ for which there exists a cofinal $\prec$-chain $(A_i, n_i)_{i \in \mathbb{N}}$ in $F$ where each $A_i$ contains some prefix of $p$. Since $F$ is a non-empty countable ideal, it is clear that at least one cofinal $\prec$-chain exists in $F$. Given such a cofinal chain $(A_i, n_i)_{i \in \mathbb{N}}$, we can assume w.l.o.g. that each $A_i$ is non-empty. Also note that $(n_i)_{i\in\mathbb{N}}$ is a strictly increasing chain, and each $w\in A_i$ has an extension $w' \in A_{i+1}$ satisfying $n_i < \lambda(w')$. It follows that $\emptyset \neq N \subseteq \dom(\delta)$.

Consider $p, q \in N$. Then arbitrarily long prefixes of $p$ and $q$ appear in finite sets occurring in $F$, and these have common upper bounds. The requirement that $(A,n) \prec (B,m)$ and $u \in B$ implies $A\subseteq f(u,m)$ then tells us that every prefix $w$ of $p$ can be extended to some $q_w$ with $\delta(q_w) = \delta(q)$, and every prefix $u$ of $q$ can be extended to some $p_u$ with $\delta(p_u) = \delta(p)$. This shows that for any $U \in \mathcal{O}(\mathbf{X})$ we have that $\delta(p) \in U \Leftrightarrow \delta(q) \in U$. As admissibility implies being $T_0$, we conclude that $\delta(p) = \delta(q)$. It follows that $\{\delta(p) \mid p \in N\}$ is some singleton $\{x\}$. Let $(A,n) \in F$ and $w \in A$. By definition of cofinality, there is some $(B,m) \succ (A,n)$ occurring in a cofinal chain. We know that any $v \in B$ is extendible to a name for $x$, and that $w \in f(v,m)$, hence $w$ is extendible to a name of $x$, too. This shows $F \subseteq F_x$.

To show that $F_x \subseteq F$, assume $(A,n)\in F_x$ and fix any $p\in N$. Since $A\subseteq U_x$ and $\delta(p)=x$, for each $w\in A$ there is a prefix $v$ of $p$ such that $w\in f(v,m)$ for some $m$. Since $p\in N$, using the monotonicity of $f$ and the fact that $A$ is finite, it follows that there is some $(B,m)\in F$ with $n<m$ such that $B$ contains a prefix $v$ of $p$ long enough to satisfy $A\subseteq f(v,m)$. Next let $(B',m')$ be any immediate $\prec$-successor of $(B,m)$ in $F$, and note that every $w\in A$ has an extension $w' \in B'$. Finally, let $(B'',m'')$ be any immediate $\prec$-successor of $(B',m')$ in $F$. Clearly $B''\not=\emptyset$ and $n < m''$ and $n < \lambda(w)$ for each $w\in B''$. Furthermore, $A\subseteq f(u,m'')$ for each $u\in B''$, because each $w\in A$ has an extension $w'\in B'$, $B' \subseteq f(u,m'')$, and $f(u,m'')$ is closed under prefixes. Finally, if $w\in f(u,n)$ for some $u\in A$, then $w\in f(u',n)$ for any extension $u'\in B$ of $u$, hence $w$ has an extension $w'$ in $B''$. Therefore, $(A,n) \prec (B'',m'') \in F$.

We conclude that $F = F_x$, and that $\mathbf{X}$ is computably isomorphic to $\mathrm{RI}(\prec)$.
\end{proof}
\end{theorem}

\begin{definition}
\label{def:precomputableQP}
We call a space $\mathbf{X}$ satisfying the equivalent criteria of Theorem \ref{theo:precomputableQP} a \emph{precomputable quasi-Polish space}.
\end{definition}

\begin{theorem}
\label{theo:computableQP}
The following are equivalent for a represented space $\mathbf{X}$:
\begin{enumerate}
\item $\mathbf{X}$ is  precomputably quasi-Polish and effectively separable.
\item $\mathbf{X}$ is  precomputably quasi-Polish and computably overt.
\item $\mathbf{X}$ admits an effectively fiber-overt computably admissible total representation $\delta$.
\item There exists a c.e.~transitive relation $\mathalpha{\prec} \subseteq \mathbb{N} \times \mathbb{N}$ such that $\mathbf{X} \cong \mathrm{RI}(\prec)$, such that $\prec$ has a c.e.~extendability predicate $E$.
\end{enumerate}
\begin{proof}
\begin{description}
\item[$1. \Rightarrow 2.$] Effective separability trivially implies computable overtness.
\item[$2. \Rightarrow 3.$] From Theorem \ref{theo:precomputableQP} we obtain an effectively fiber-overt computably admissible representation $\delta'$ with $\Pi^0_2$-domain. As the preimage of computable overt $\mathbf{X}$ under the effectively fiber-overt $\delta'$, we see that we can also obtain $\overline{\dom(\delta')} \in \mathcal{V}(\Baire)$. We apply Lemma \ref{lemma:trace} trace to obtain computable, computably invertible and effectively open $f : \Baire \to \Baire$ such that $\delta = \delta' \circ f$ is the desired total representation.
\item[$3. \Rightarrow 1.$] Using Theorem \ref{theo:precomputableQP} and that spaces with total representations inherit effective separability from $\Baire$.
\item[$2. \Leftrightarrow 4.$] A space of the form $\mathrm{RI}(\prec)$ is computably overt iff the extendability predicate for $\prec$ is computably enumerable.
\end{description}
\end{proof}
\end{theorem}

\begin{definition}
\label{def:computableqp}
We call a space $\mathbf{X}$ satisfying the equivalent criteria of Theorem \ref{theo:computableQP} a \emph{computable quasi-Polish space}.
\end{definition}

\begin{remark}
Criterion 2 of Theorem \ref{theo:precomputableQP} is proposed as the definition of a ``computable quasi-Polish space'' by M.~Korovina and O.~Kudinov in \cite[Definition 7]{korovina2}. We prefer to include computable overtness in the definition of \emph{computable quasi-Polish} (see Theorem \ref{theo:computableQP} and Definition \ref{def:computableqp}) as this property is often useful for applications, and present in all natural examples. This also mirrors the terminology for metric spaces, where a computable Polish space is by definition computably overt, rather than being merely computably completely metrizable. Therefore, our notion of computable quasi-Polish essentially corresponds to an ``effectively enumerable computable quasi-Polish space'' in the terminology of \cite{korovina2}.

V.~Selivanov \cite{selivanov8} suggested the possibility of using an effective version of a \emph{convergent approximation space} \cite{becher}, which is closely related to Criterion 4 of Theorem~\ref{theo:computableQP}.
\end{remark}

\begin{lemma}
\label{lemma:trace}
From non-empty $A \in \boldp^0_2(\Baire)$ and $\overline{A} \in \mathcal{V}(\Baire)$ we can compute $f : \Baire \to \Baire$ with $f[\Baire] = A$, such that $f$ is open and has a computable inverse.
\begin{proof}
We take the $\Pi^0_2$-information about $A$ via some monotone function $\lambda : \mathbb{N}^* \to \mathbb{N}$ such that $p \in A$ iff $\{\lambda(p_\leq n) \mid n \in \mathbb{N}\}$ is unbounded. We take the overt information about $A$ as an enumeration of all $w \in \mathbb{N}^*$ that are extendible to an element of $A$. Call $w \in \mathbb{N}^*$ \emph{productive}, if $\lambda(w) > \lambda(w_{|w|-1})$ and $w$ is enumerated as extending to a member of $A$. Let the empty word $\varepsilon$ be productive by convention. Clearly, we can enumerate all productive words, and each productive word has some productive extensions.

We construct $f$ as the limit of a monotone function $F : \mathbb{N}^* \to \mathbb{N}^*$, which we define in turn by induction on the length of the input. The range of $F$ will be exactly the productive words. Set $F(\varepsilon) = \varepsilon$. If $F(w) = u$, then we search for productive extensions of $u$. We can enumerate those as $(u_n)_{n \in \mathbb{N}}$ (in particular, there are some). We then extend $F(wn) = u_n$. It is straight-forward to check that the construction of $F$ gives the desired properties to $f$.
\end{proof}
\end{lemma}

\begin{observation}
If $\mathbf{X}$ is (pre)computably quasi-Polish and $f : \mathbf{X} \to \mathbf{Y}$ is computable, effectively open and surjective, then $\mathbf{Y}$ is (pre)computably quasi-Polish.
\end{observation}

\begin{observation}
For a computable metric space $\mathbf{X}$ the following are equivalent:
\begin{enumerate}
\item $\mathbf{X}$ is Polish.
\item $\mathbf{X}$ is precomputably quasi-Polish.
\item $\mathbf{X}$ is computably quasi-Polish.
\end{enumerate}
\end{observation}

\subsection{Computability of overt choice}
We first show that overt choice is computable for every precomputable quasi-Polish space. We then prove a partial converse, that if overt choice is computable (even with respect to some oracle) for a countably based $T_1$-space $\mathbf{X}$, then $\mathbf{X}$ is quasi-Polish (but not necessarily precomputable quasi-Polish).

\begin{theorem}
\label{theo:overtchoicequasipolish}
Overt choice $\VC_\mathbf{X}$ is computable for every precomputable quasi-Polish space $\mathbf{X}$.
\end{theorem}
\begin{proof}
Let $\prec\subseteq \NN\times\NN$ be a c.e. transitive relation such that $\mathbf{X} \cong \mathrm{RI}(\prec)$. For $n\in\NN$ we write ${\uparrow}n $ for the basic open subset of $\mathrm{RI}(\prec)$ consisting of all rounded ideals that contain $n$. Clearly $\cal{V}(\mathrm{RI}(\prec)) \cong \cal{V}(\mathbf{X})$, so it suffices to consider overt choice for $\mathrm{RI}(\prec)$.

Given a presentation for some non-empty closed $A\in\cal{V}(\mathrm{RI}(\prec))$, we construct a $\prec$-ascending chain $(n_i)_{i\in\NN}$ such that $A$ intersects each basic open ${\uparrow}n_i$. To construct the chain, first choose any $n_0\in \NN$ such that $A$ has non-emtpy intersection with ${\uparrow}n_0$. Once $n_i$ has been decided, choose any $n_{i+1}\in \NN$ such that $A$ intersects ${\uparrow}n_{i+1}$ and $n_i \prec n_{i+1}$. Such a chain can be computed from a presentation of $A$ because $\prec$ is a c.e. relation and it can be semi-decided whether $A$ intersects a given basic open set.

Finally, we can enumerate the set $I = \{ n\in\NN \mid (\exists i\in\NN)\, n \prec n_i \}$, which is the rounded ideal generated by the sequence $(n_i)_{i\in\NN}$. For any $n\in\NN$ with $I \in {\uparrow}n$, there is $i\in\NN$ with $n \prec n_i \in I$, hence $I \in {\uparrow n_i} \subseteq {\uparrow n}$. Therefore, $I \in A$ because every basic open containing $I$ intersects $A$.
\end{proof}

We obtain the following corollary, which generalizes the corresponding theorem for computable Polish spaces from \cite{presser}:

\begin{corollary}
Let $\mathbf{X}$ be a precomputable quasi-Polish space. The computable map $(a_i)_{i \in \mathbb{N}} \mapsto \mathrm{cl} \{a_i \mid i \in \mathbb{N}\} : \mathbf{X}^\mathbb{N} \to \mathcal{V}(\mathbf{X}) \setminus \{\emptyset\}$ has a computable multi-valued inverse.
\begin{proof}
Given non-empty $A \in \mathcal{V}(\mathbf{X})$ we can enumerate all basic open sets $U_i$ having a non-empty intersection with $A$. We can then compute $\mathrm{cl} (U_i \cap A) \in \mathcal{V}(\mathbf{X})$, and use $\mathrm{VC}_\mathbf{X}$ to extract a point. The resulting sequence is dense in $A$.
\end{proof}
\end{corollary}

\subsection{Continuity of overt choice as a completeness notion}
\label{subsec:completeness}
We next prove a partial converse to Theorem~\ref{theo:overtchoicequasipolish} using a game theoretic characterization of quasi-Polish spaces.

Given a non-empty space $\mathbf{X}$, the \emph{convergent strong Choquet game} \cite{dorais,debrecht6} is played as follows. Player I first plays a pair $(U_0, x_0)$ with $U_0 \in \mathcal{O}(\mathbf{X})$ and $x_0 \in U_0$. Player II must respond with an open set $V_0$ such that $x_0 \in V_0 \subseteq U_0$. Player I then responds with a pair $(U_1, x_1)$ with $U_1$ open and $x_1 \in U_1 \subseteq V_0$, then Player II must play an open $V_1$ with $x_1 \in V_1 \subseteq U_1$, and so on. Player II wins the game if and only if the sequence of opens $(V_i)_{i\in\NN}$ is a neighborhood basis for a unique point in $\mathbf{X}$. It was shown in \cite{debrecht6} that a non-empty countably based $T_0$-space $\mathbf{X}$ is quasi-Polish iff Player II has a winning strategy (see also \cite{ruiyuan}, which fills a gap in the original proof).

\begin{theorem}
\label{theo:vcimpliesqp}
If $\mathbf{X}$ is a countably based $T_1$-space and $\VC_\mathbf{X}$ is continuous, then $\mathbf{X}$ is quasi-Polish.
\begin{proof}
Let $R$ be a continuous realizer for $\VC_\mathbf{X}$. We show that $R$ can be used to define a winning strategy for Player II in the convergent strong Choquet game for $\mathbf{X}$. The basic idea of Player II's strategy is to present to $R$ the closure of the sequence of elements $(x_0, x_1, \ldots, x_i)$ played by Player I, and the $V_i$ played by Player II will correspond to the output of $R$.

Fix a countable basis $(B_k)_{k\in\NN}$ for $\mathbf{X}$. A valid input to $R$ consists of an enumeration of all the $B_k$ that intersect some non-empty closed $A \subseteq \mathbf{X}$. We can assume the output of $R$ will be a decreasing sequence of basic opens forming a neighborhood basis for some $x\in A$.

At each round $i$, Player II will keep track of a finite set $A_i$ (which is closed because $\mathbf{X}$ is $T_1$). Set $A_0 = \emptyset$. (In round $i$, the set $A_i$
will actually be the finite set of elements that have been played by Player I that are distinct from the element $x_i$ played that round.)

At round $i$, Player I plays $(U_i, x_i)$. Up until now, Player II's strategy will have guaranteed that the following all hold at each round $i$:
\begin{enumerate}
\item $x_i \not\in A_i$,
\item $U_i \cap A_i = \emptyset$,
\item the information fed to $R$ until now is consistent with a presentation for $A_i \cup \{x_i\}$,
\item the output of $R$ until now is consistent with a presentation of $x_i$,
\item the output of $R$ until now can only be extended to a presentation of an element in $\mathbf{X} \setminus A_i$.
\end{enumerate}

(This is trivial for $i=0$, before any information is fed to the realizer $R$).

Player II chooses $V_i$ as follows. There are two cases:
\begin{quote}
\begin{description}
\item[Case 1)] Either $i=0$ or $x_i = x_{i-1}$. Then define $A_{i+1} = A_i$.

\item[Case 2)] $x_i \not= x_{i-1}$. Then define $A_{i+1} = A_i \cup \{ x_{i-1} \}$.
\end{description}
\end{quote}

In either case, extend the presentation being fed to $R$ so that it is a presentation of $A_i \cup \{x_i\}$ (this is possible by item 3 above). Player II makes
sure the presentation is extended enough so that it includes every basic open $B_k$ that intersects $A_i \cup \{x_i\}$ for each $k \leq i$ (this is to guarantee that
as $i$ goes to infinity we are actually giving $R$ a valid presentation).

Since $R$ is being fed a presentation of $A_i \cup \{x_i\}$, items 1 and 5 above force the output of $R$ to be a presentation of $x_i$. So $R$ must eventually
output a basic open $W$ with $x_i \in W \subseteq (U_i \setminus A_{i+1})$. At this point we pause our execution of $R$ and do not feed it any more information. Define $V_i$ to be the intersection of $W$ with all of the basic opens $B_k$ that contain $x_i$ that have been fed as input to $R$ so far. Player II plays $V_i$, and the game continues to round $i+1$.

We check that the items 1) - 5) still hold in round $i+1$. Since we must have $x_{i+1} \in U_{i+1} \subseteq V_i$, and $V_i \cap A_{i+1} = \emptyset$, items 1 and 2 hold. Similarly, items 4 and 5 hold because $x_{i+1} \in U_{i+1} \subseteq W$ and $W \cap A_{i+1} = \emptyset$, where $W$ is the last basic open that was outputted by
$R$. Concerning item 3, the only subtlety is if Case 2) held in round $i$ and $x_{i+1} \neq x_i$. So $x_i \not\in A_{i+1}$ but so far $R$ has been presented
with information for $A_{i+1} \cup \{x_i\}$. However, item 3 still holds in round $i+1$ because $x_{i+1}$ will be in $V_i$ which is a subset of the intersection of
all the basic opens containing $x_i$ that have so far been fed to $R$. Therefore, any basic open $B_k$ that has been fed as input to $R$ either intersects
$A_{i+1}$ or else contains $V_i$, hence contains $x_{i+1}$. Therefore, the presentation given to $R$ is consistent with a presentation for $A_{i+1} \cup
\{x_{i+1}\}$. This shows that Player II's strategy is well-defined.

Finally, we must show that this strategy is winning. It suffices to show that the information fed to $R$ is a valid presentation of the overt closed set $A$
defined as the closure of the infinite sequence $(x_i)_{i\in\NN}$, because then the sequence $V_i$ will be an open neighborhood basis of the point in $A$ chosen by $R$.

Clearly, every basic open that was presented to $R$ intersects $A$. So to prove that the presentation given to $R$ is valid, it only remains to check that if
$B_k$ is a basic open that intersects $A$, then $B_k$ was included in the presentation to $R$ at some time. But if $B_k$ intersects $A$, then $B_k$ contains some
$x_i$, and so for round $i' > \max \{i,k\}$ we have that $B_k$ intersects $A_{i'} \cup \{x_{i'}\}$, and so $B_k$ was included in the presentation to $R$ in round
$i'$.
\end{proof}
\end{theorem}

\begin{corollary}
\label{corr:overtchoiceqpcharac}
A countably-based $T_1$-space $\mathbf{X}$ is quasi-Polish if and only if $\VC_\mathbf{X}$ is continuous.
\begin{proof}
Combine Theorem \ref{theo:vcimpliesqp} with the relativization of Theorem \ref{theo:overtchoicequasipolish}.
\end{proof}
\end{corollary}

A classical result by E.~Michael \cite{michael} states that if $X$ is a zero-dimensional metrizable space and $Y$ is a complete metric space, then every lower semi-continuous function from $X$ to the non-empty closed subsets of $Y$ admits a continuous selection. It was then shown in \cite{vanMill} that the completeness of $Y$ is necessary. If $X$ is a separable zero-dimensional metrizable space and $Y$ is a $\mathrm{QCB}_0$-space, then any lower semi-continuous function $F$ from $X$ to the closed subsets of $Y$ can be viewed as a continuous function from $X$ to ${\cal V}(Y)$. Since there exists a continuous reduction of $F$ to an admissible representation of ${\cal V}(Y)$, it is clear that if overt choice on $Y$ is continuous then $F$ has a continuous selection. Conversely, a continuous solution to overt choice on $Y$ is equivalent to the existence of a continuous selection for the admissible representation of ${\cal V}(Y)$. Therefore, we can view Corollary~\ref{corr:overtchoiceqpcharac} as an extension of these classical selection results to the case that $Y$ is a countably based (possibly non-metrizable) $T_1$-space.

Note that a computable version of Corollary \ref{corr:overtchoiceqpcharac} does not hold, because the computability of  $\VC_\mathbf{X}$ does not imply that $\mathbf{X}$ is even a precomputable quasi-Polish space. A trivial counter example is the singleton space $\{p\}$ where $p \in \Cantor$ is chosen such that $p \nleq_{\mathrm{M}} A$ for any non-empty $\Pi^0_2$-set $A$. Then $p$ can be computed from any non-empty $A \in \mathcal{V}(\{p\})$, even though $\{p\}$ is not precomputably quasi-Polish as it violates Condition 3 of Theorem \ref{theo:precomputableQP} by design.


\section{Other countably-based spaces}
\label{sec:other}
In this section, we study overt choice on countably-based spaces that are not quasi-Polish. We do not know of relevant examples in this class where overt choice would be computable, and as such our focus is on investigating the Weihrauch degrees of overt choice of such spaces. First we gather some auxiliary on Weihrauch reducibility in Subsection \ref{subsec:qb:aux}. The main content of this section is spread over Subsection \ref{subsec:qb:general} where we provide results pertaining to spaces fulfilling various general conditions, and Subsection \ref{subsec:qb:specific}, where we consider overt choice for two specific spaces, namely the Euclidean rationals $\mathbb{Q}$ and the space $S_0$ from \cite{debrecht8}. Plenty of questions are left open, and we state some of those in Subsection \ref{subsec:qb:questions}.

\subsection{Some auxiliary results}
\label{subsec:qb:aux}
We consider multivalued functions $f : \subseteq \mathcal{O}(\mathbb{N}) \mto \mathbf{X}$. Call such $f$ \emph{uncomputable everywhere} if for every finite $A \subseteq \mathbb{N}$ the restriction $f|_{\{U \in \mathcal{O}(\mathbb{N}) \mid A \subseteq U\}}$ is uncomputable.

\begin{proposition}
\label{prop:notbelowcn}
Let $f : \subseteq \mathcal{O}(\mathbb{N}) \mto \mathbf{X}$ be uncomputable everywhere, and satisfy $\mathbb{N} \in \dom(f)$. Then $f \nleqW \C_\mathbb{N}$.
\begin{proof}
Assume that $f \leqW \C_\mathbb{N}$ via some computable $K$, $H : \subseteq \Baire \to \Baire$. The semantics for $H$ mean that if $p$ is some enumeration of the input to $f$, then $H(p)$ is some enumeration of the input to $\C_\mathbb{N}$. We will construct some $p \in \dom(H)$ with $\range (p) = \mathbb{N}$ and $\range (H(p)) = \mathbb{N}$, i.e.~$p$ is a name for some valid input to $f$, but $H(p)$ is not a name for some valid input to $\C_\mathbb{N}$, and thus derive a contradiction.

Let $h : \subseteq \mathbb{N}^* \to \mathbb{N}^*$ be a word function for $H$. Assume that $\forall p \in \dom(H) \ 0 \notin \range(H(p))$. Then $0$ is a valid output to $\C_\mathbb{N}$ on any set represented by $H(p)$, which would imply computability of $f$. Thus, there is some finite prefix $w_0 \in \mathbb{N}^*$ such that $0 \in B_0 := \range( h(w_0))$.  As before, assuming that $\forall p \in \dom(H) \cap w_00\Baire \ 1 \notin \range(H(p))$ leads to a contradiction of $f$ being uncomputable everywhere, so we can extend to some $w_00w_1$ such that $1 \in B_1 := \range (h(w_00w_1))$, and so on. Let $p := w_00w_11w_22\ldots$. This $p$ is the desired contradictory input to $H$.
\end{proof}
\end{proposition}

\begin{proposition}
\label{prop:closedchoicetransfer}
Let $\mathbf{Y}$ be a computable Hausdorff space. If $f : \subseteq \mathbf{X} \mto \mathbf{Y}$ satisfies $f \leqW \C_\Cantor \star g$, then also $f \leqW \C_\mathbf{Y} \star g$.
\begin{proof}
The reduction $f \leqW \C_\Cantor \star g$ provides us with, for any $x \in \dom(f)$, a tree $T_x \subseteq \{0,1\}^*$ and a continuous function $K_x : [T_x] \to \mathbf{Y}$ such that $K_x([T_x]) \subseteq f(x)$. We can compute $T_x$ and $K_x$ jointly with a single application of $g$, which also yields all the other information from $g$ we may desire. From $T_x$ and $K_x$ we can compute $K_x([T_x]) \in \mathcal{A}(\mathbf{Y})$ by noting that $y \notin K_x([T_x])$ iff $K_x^{-1}(\mathbf{X} \setminus \{y\}) \supseteq T_x$. The argument follows.
\end{proof}
\end{proposition}

\begin{proposition}
\label{prop:cqcn}
$\C_\mathbb{Q} \equivW \C_\mathbb{N}$
\begin{proof}
That $\C_\mathbb{N} \leqW \C_\mathbb{Q}$ follows from $\mathbb{N}$ embedding as a computably closed subspace into $\mathbb{Q}$ by \cite[Corollary 4.3]{paulybrattka}, and that $\C_\mathbb{Q} \leqW \C_\mathbb{N}$ follows from the existence of a computable surjection $s : \mathbb{N} \to \mathbb{Q}$ and \cite[Proposition 3.7]{paulybrattka}.
\end{proof}
\end{proposition}

\begin{corollary}
\label{corr:crreduction}
If $f : \subseteq \mathbf{X} \mto \mathbb{Q}$ satisfies $f \leqW \C_\mathbb{R}$, then already $f \leqW \C_\mathbb{N}$.
\begin{proof}
As shown in \cite{paulybrattka} we have $\C_\mathbb{R} \equivW \C_\Cantor \star \C_\mathbb{N}$. We can thus apply Proposition \ref{prop:closedchoicetransfer} to conclude that $f \leqW \C_\mathbb{Q} \star \C_\mathbb{N}$, and then use Proposition \ref{prop:cqcn} together with $\C_\mathbb{N} \equivW \C_\mathbb{N} \star \C_\mathbb{N}$ from \cite{paulybrattka}.
\end{proof}
\end{corollary}
\subsection{General observations}
\label{subsec:qb:general}
Our first result shows that overt choice for countably-based spaces is never able to provide non-computable discrete information:
\begin{proposition}
\label{prop:cbelimination}
Let $\mathbf{X}$ be effectively countably-based. If $f : \subseteq \Cantor \mto \mathbb{N} \leqW \mathrm{VC}_\mathbf{X}$, then $f$ is computable.
\begin{proof}
We pick a standard countable basis $(U_n)_{n \in \mathbb{N}}$ for $\mathbf{X}$, and assume $\mathbf{X}$ to be represented by the associated standard representation. Consider the computable outer reduction witness $K : \subseteq \Cantor \times \Baire \mto \mathbb{N}$. This gives rise to a computable sequence $(n_i,w_i,k_i)_{i \in \mathbb{N}}$ such that $K(p,q)$ can return $n$ iff $\exists i \in \mathbb{N} \ n = n_i \wedge \delta(q) \in U_{k_i} \wedge w_i \prec p$. Let $H : \subseteq \Cantor \to \mathcal{V}(\mathbf{X})$ be the inner reduction witness. Given $p \in \dom(f)$, start testing for all $i \in \mathbb{N}$ in parallel whether $U_{k_i}$ intersects $H(p) \in \mathcal{V}(\mathbf{X})$ and $w_i \prec p$. This has to be true for at least one $i \in \mathbb{N}$, and once we have found a suitable candidate, we can return $n_i$ as a correct output to $f(p)$.
\end{proof}
\end{proposition}

\begin{corollary}
For effectively countably-based $\mathbf{X}$, its overt choice $\mathrm{VC}_\mathbf{X}$ is not $\omega$-discriminative (in the sense of \cite{kreuzer}).
\end{corollary}

Besides the degrees of problems inspired by computability theory (which would often return Turing degrees as outputs), the investigations of specific Weihrauch degrees so far have not yet encountered non-computable yet not $\omega$-discriminative degrees. We thus see that for countably-based spaces with non-computable overt choice principles we find ourselves in an unexplored region of the Weihrauch lattice. We can go even further for sufficiently homogeneous spaces:

\begin{corollary}
\label{corr:pipevcx}
Let $\mathbf{X}$ be effectively countably-based such that every non-empty open subset contains a copy of $\mathbf{X}$. If $\VC_\mathbf{X}$ is non-computable, then $\VC_\mathbf{X} \pipeW \C_\mathbb{N}$.
\begin{proof}
 We can identify $A \in \mathcal{V}(\mathbf{X})$ with $\{n \in \mathbb{N} \mid I_n \cap A \neq \emptyset\} \in \mathcal{O}(\mathbb{N})$ for some canonical basis $(I_n)_{n \in \mathbb{N}}$ of $\mathbf{X}$. Since $\mathbf{X}$ by assumption embeds into any of its non-trivial basic open sets, we find that the result map $\VC_\mathbf{X} : \subseteq \mathcal{O}(\mathbb{N}) \mto \mathbf{X}$ is even non-computable everywhere, and hence Proposition \ref{prop:notbelowcn} lets us conclude $\VC_\mathbf{X} \nleqW \C_\mathbb{N}$. That $\C_\mathbb{N} \nleqW \VC_\mathbf{X}$ follows from Proposition \ref{prop:cbelimination}.
\end{proof}
\end{corollary}

We can actually obtain some upper bounds for overt choice on countably-based spaces. Recall that $\Pi^0_2\C_\mathbf{X}$ takes as input a non-empty $\Pi^0_2$-subset of $\mathbf{X}$ (coded in the usual way via an appropriate Borel code), and outputs an element of that set.

\begin{proposition}
\label{prop:vcandpi2}
Let $s : \mathbf{X} \to \mathbf{Y}$ be a computable surjection, and $\mathbf{Y}$ be effectively countably-based. Then $\VC_\mathbf{Y} \leqW \Pi^0_2\C_\mathbf{X}$.
\begin{proof}
Let $(U_n)_{n \in \mathbb{N}}$ be an effective countable basis of $\mathbf{Y}$. Given $A \in \mathcal{V}(\mathbf{Y})$ we can compute $\{x \in \mathbf{X} \mid \forall n \in \mathbb{N} \ s(x) \in U_n \Rightarrow U_n \cap A \neq \emptyset\} \in \Pi^0_2(\mathbf{X})$, apply $\Pi^0_2\C_\mathbf{X}$ to obtain an element $x_0$ of that set, and then notice that $s(x_0) \in A$.
\end{proof}
\end{proposition}

\begin{corollary}
\label{corr:cbaireupper}
Let $\mathbf{X}$ be a $\Sigma^1_1$-subspace of $\mathcal{O}(\mathbb{N})$. Then relative to some oracle it holds that $\VC_\mathbf{X} \leqW \C_\Baire$.
\begin{proof}
A $\Sigma^1_1$-subspace is the range of a continuous surjection from $\Baire$. This is computable relative to some oracle, and the relativization of Proposition \ref{prop:vcandpi2} then gives $\VC_\mathbf{X} \leqW \Pi^0_2\C_\Baire$. That $\Pi^0_2\C_\Baire \equivW \C_\Baire$ is straight-forward, it was observed e.g.~in \cite{pauly-kihara4}.
\end{proof}
\end{corollary}

\begin{corollary}
\label{corr:vcqpi02}
Let $\mathbf{X}$ be effectively countable and effectively countably-based. Then $\VC_\mathbf{X} \leqW \Pi^0_2\C_\mathbb{N}$.
\end{corollary}

\subsection{Overt choice for specific spaces}
\label{subsec:qb:specific}
We consider overt choice for two specific countably-based yet not quasi-Polish spaces. The first space is $\mathbb{Q}$, seen as a subspace of $\mathbb{R}$. The second specimen is the space $S_0$ defined as follows:

\begin{definition}
The underlying set of $S_0$ is $\mathbb{N}^*$, the set of finite sequences of natural numbers (including the empty word $\varepsilon$). The topology is generated by the sets $\{\{u \in \mathbb{N}^* \mid u \nsucceq w\} \mid w \in \mathbb{N}^*\}$. This means membership in a basic open set provides the knowledge of finitely many words which are not a prefix of the given point.
\end{definition}

Our choice of studying these specific spaces is not completely arbitrary: They both belong to the four canonic counter-examples for being quasi-Polish (meaning that a coanalytic subspace of $\mathcal{O}(\mathbb{N})$ is either quasi-Polish or contains a $\boldp^0_2$ copy of one of the canonic counter-examples) \cite{debrecht8}. Given our interest in the continuity of overt choice as a completeness notion, this makes these spaces high priority targets for classification.

Since $\mathbb{Q}$ respectively $S_0$ is computably isomorphic $\mathbb{Q} \times \mathbb{Q}$ respectively $S_0 \times S_0$, we can conclude the following from Corollary \ref{corr:vcproducts}:

\begin{corollary}
$\VC_\mathbb{Q} \equivW \VC_\mathbb{Q}^*$ and $\VC_{S_0} \equivW \VC_{S_0}^*$.
\end{corollary}

Overt choice on $\mathbb{Q}$ was already studied by Brattka in \cite{brattka14} (without having the modern terminology available) and shown to be uncomputable. Together with the result from \cite{presser} that overt choice on Polish spaces is continuous, and the Hurewicz dichotomy stating that a coanalytic separable metric space is either Polish or has a copy of $\mathbb{Q}$ as a closed subspace, it already follows that:

\begin{corollary}[\footnote{Note that Corollary \ref{corr:overtchoiceqpcharac} above shows that the restriction to coanalytic spaces is not actually necessary here.}]
For a coanalytic separable metric space $\mathbf{X}$ we find that $\VC_\mathbf{X}$ is continuous iff $\mathbf{X}$ is Polish.
\end{corollary}

The starting point of our investigation of the degree of $\VC_\mathbb{Q}$ will be to introduce a somewhat more accessible problem defined on trees. We say that a tree $T \subseteq \{0,1\}^*$ has eventually constant paths everywhere, if for each $w \in T$ there is $u \succ w$ and $b \in \{0,1\}$ such that for all $n \in \mathbb{N}$ we have $ub^n \in T$. In words, every vertex in the tree can be extended to a path that eventually goes always right or always left. The principle $\mathrm{ECP}$ has to find such an eventually constant path from an enumeration of the tree.

\begin{definition}
Let $\mathrm{ECP} : \subseteq \mathcal{O}(\{0,1\}^*) \mto \Cantor$ be defined by $T \in \dom(\mathrm{ECP})$ iff $T \neq \emptyset$ has eventually constant paths everywhere, and $p \in \mathrm{ECP}(T)$ if $p = w0^\omega$ or $p = w1^\omega$ for some $w \in \{0,1\}$ and $\forall n \in \mathbb{N} \ p_{\leq n} \in T$.
\end{definition}

\begin{proposition}
\label{prop:ecpnc}
$\mathrm{ECP}$ is not computable.
\begin{proof}
We describe a strategy how to construct an input on which a putative algorithm fails. Start by enumerating longer and longer prefixes of $01^\omega$. If the algorithm does not eventually output $01$, then continuing to output all prefixes of $01^\omega$ makes the algorithm fail. If the algorithm output $01$ at the moment where the longest prefix enumerated so far is $01^{k_0}$, then we enumerate all $0^n$, as well as longer and longer prefixes of $01^{k_0}0^\omega$. The algorithm has to output $01^{k_0}0$ eventually. At that moment, we enumerate all prefixes of $01^\omega$, and start enumerating prefixes of $01^{k_0}0^{k_1}1^\omega$, where $01^{k_0}0^{k_1}$ is the longest prefix of $01^{k_0}0^\omega$ we enumerated so far. Continuing this process will force the algorithm to either deviate at some stage and fail, or to output a sequence with infinitely many alternations between $0$ and $1$, and hence fail, too.
\end{proof}
\end{proposition}

\begin{proposition}
\label{prop:ecpvcq}
$\mathrm{ECP} \equivW \mathrm{VC}_\mathbb{Q}$.
\begin{proof}
We construct an overt subset of $[0,1] \cap \mathbb{Q}$ from the enumeration of the tree. Let $(q_n)_{n \in \mathbb{N}}$ be some standard enumeration of $[0,1] \cap \mathbb{Q}$. We construct a basis $(B_w)_{w \in \{0,1\}^*}$ of $[0,1] \cap \mathbb{Q}$ as follows: $B_\varepsilon := [0,1] \cap \mathbb{Q}$. Once $B_{wb} = (a,b) \cap \mathbb{Q}$ is defined, pick some irrational $\tau \in (\frac{2}{3}a + \frac{1}{3}b, \frac{1}{3}a + \frac{2}{3}b)$. Of the intervals $(a,\tau) \cap \mathbb{Q}$ and $(\tau,b) \cap \mathbb{Q}$ one contains the least rational (w.r.t. $(q_n)_{n \in \mathbb{N}}$). If it is $(a,\tau) \cap \mathbb{Q}$ ,then $B_{wbb} = (a,\tau) \cap \mathbb{Q}$ and $B_{wb\overline{b}} = (\tau,b) \cap \mathbb{Q}$ (here $\overline{b}$ denotes the complementary bit to $b$), otherwise the intervals are assigned in reversed roles.

The neighborhood filter of some $q \in \mathbb{Q} \cap [0,1]$ now corresponds to an eventually constant element $p \in \Cantor$, and an overt set in $\mathcal{V}(\mathbb{Q} \cap [0,1])$ corresponds to an enumeration of a tree where each vertex can be extended into an eventually constant path. This constitutes the desired equivalence.
\end{proof}
\end{proposition}

\begin{corollary}
\label{corr:vcqcn}
$\VC_\mathbb{Q} \pipeW \C_\mathbb{N} \equivW \C_\mathbb{Q}$.
\begin{proof}
By Propositions \ref{prop:ecpnc}, \ref{prop:ecpvcq} $\VC_\mathbb{Q}$ is not computable, and we can thus apply Corollary \ref{corr:pipevcx} to conclude $\VC_\mathbb{Q} \pipeW \C_\mathbb{N}$. That $\C_\mathbb{N} \equivW \C_\mathbb{Q}$ is Proposition \ref{prop:cqcn}.
\end{proof}
\end{corollary}

\begin{corollary}
\label{corr:cvqncr}
$\VC_\mathbb{Q} \nleqW \C_\mathbb{R}$.
\begin{proof}
Combine Corollary \ref{corr:vcqcn} and Corollary \ref{corr:crreduction}.
\end{proof}
\end{corollary}

Before moving on from $\VC_\mathbb{Q}$ to $\VC_{S_0}$ we shall examine the content of Brattka's proof from \cite{brattka14} that $\VC_\mathbb{Q}$ is non-computable. From his construction we extract the following definition:

\begin{definition}
Let $\mathrm{HitSparse} :\subseteq \mathcal{A}(\mathbb{N}) \times \Baire \to \mathcal{O}(\mathbb{N})$ be defined as follows:
\begin{itemize}
\item $(A,f) \in \dom(\mathrm{HitSparse})$ if $A$ is infinite, and
\item $U \in \mathrm{HitSparse}(A,f)$ if $U \cap A \neq \emptyset$ and $\forall n \in \mathbb{N} \ |[n,f(n)] \cap U| \leq 1$
\end{itemize}
\end{definition}

The intuition is that we are trying to solve the usual discrete choice $\C_\mathbb{N}$, but are allowed to make infinitely many guesses. These guesses, however, have to be sparse -- to make up for that, we assume that there are actually infinitely many correct solutions (on its own this requirement has no impact on the degree of $\C_\mathbb{N}$.

\begin{theorem}[Brattka \cite{brattka14}]
\begin{enumerate}
 \item $\mathrm{HitSparse}$ is not computable.
 \item $\mathrm{HitSparse} \leqW \VC_\mathbb{Q}$
 \end{enumerate}
 \end{theorem}

 Since trivially $\mathrm{HitSparse} \leqW \C_\mathbb{N}$, from Corollary \ref{corr:vcqcn} it follows that:
 \begin{corollary}
 $\mathrm{HitSparse} \leW \VC_\mathbb{Q}$
 \end{corollary}

To make $\VC_{S_0}$ more accessible, we again introduce a problem on trees. This time, we need a new represented space $\mathcal{S}\{0,1\}^*$ of finite sequences via the representation $\delta_\mathcal{S}$ defined inductively as $\delta_\mathcal{S}(0^\omega) = \varepsilon$, $\delta_\mathcal{S}(00p) = \delta_\mathcal{S}(11p) = \delta_\mathcal{S}(p)$, $\delta_\mathcal{S}(10p) = 0\delta_\mathcal{S}(p)$ and $\delta_\mathcal{S}(01p) = 1\delta_\mathcal{S}(p)$. Intuitively, if we are given $w \in \mathcal{S}\{0,1\}^*$ we never know for sure that we have seen the end of the finite sequence, for it can always be extended again.

\begin{definition}
Let $\mathrm{FSL} : \subseteq \mathcal{O}(\{0,1\}^*) \mto \mathcal{S}\{0,1\}^*$ be defined by $T \in \dom(\mathrm{FSL})$ if $T$ is a non-empty tree such that there exists a leaf below any vertex, and $w \in \mathrm{FSL}(T)$ if $w$ is a leaf of $T$.
\end{definition}

\begin{proposition}
\label{prop:fslnc}
$\mathrm{FSL}$ is non-computable.
\begin{proof}
We describe how to diagonalize against a hypothetical algorithm solving $\mathrm{FSL}$. The argument is essentially the same as in Proposition \ref{prop:ecpnc}.
We start off with the input $\{\varepsilon, 0\}$. The algorithm needs at some point to commit to output a leaf extending $0$. At this point, we add $1$ and $00$ to the tree. Since the new leaves are $1$ and $00$, and the algorithm can no longer output $1$, it needs to commit to $00$ eventually. At that point, we add $01$ and $000$ to the tree, and so on. Either the algorithm will at some point fail to commit to an extension of the current output, and thus output an internal vertex, or it will commit infinitely often, and thereby not output a vertex at all.
\end{proof}
\end{proposition}



\begin{proposition}
\label{prop:fslvcs}
$\mathrm{FSL} \leqW \VC_{S_0}$
\begin{proof}
We construct some $A \in \mathcal{V}(S_0)$ from the tree $T \in \dom(\mathrm{FSL})$ by iteratively updating a partial mapping $\phi : \subseteq \{0,1\}^* \to \mathbb{N}^*$ such that if $L$ is the set of leaves of our current approximation to $T$, then our current approximation to $A$ is consistent with $A = \phi[L]$. If we learn at some point that $w$ is not actually a leaf of $T$ (because it has some extension $wi \in T$), we will have given a finite amount of information about $A$ yet. In particular, there is some $N \in \mathbb{N}$ such that no mentioning of $\phi(w)N$ and $\phi(w)(N+1)$ has been given so far. This ensures that if we update our assumption that $\phi(w) \in A$ to either $\phi(w)N \in A$ or $\phi(w)(N+1) \in A$ this is consistent with all information given so far. We can thus set $\phi(w0) = \phi(w)N$ and $\phi(w1) = \phi(w)(N+1)$ without compromising our construction. The promise that there is a leaf below any vertex in $T$ ensures that any candidate put into $A$ will have a surviving candidate below it, which provides the well-definedness of $A$.

Let us assume that we are given some $u = \phi(w) \in A$ by $\VC_{S_0}$. If we knew $\phi(w) \in \mathbb{N}^*$, we could obviously reconstruct $w \in \{0,1\}^*$ and complete the reduction. However, we only know $\phi(w) \in S_0$, but only need $w \in \mathcal{S}\{0,1\}^*$. In particular, we can wait with extending our current candidate $w'$ for $w$ until we learn that $w$ indeed has an extension in $T$. But at that moment, we know the values of $\phi(w'0)$ and $\phi(w'1)$. Since at least one of them is not $\phi(w)$, we will eventually learn about $\phi(w)$ that it is either not below $\phi(w'0)$ or not below $\phi(w'1)$. But that answer then tells us how we should extend $w'$ to obtain a longer prefix of $w$.
\end{proof}
\end{proposition}

\begin{corollary}
$\mathrm{FSL} \leW \VC_{S_0}$
\begin{proof}
The reduction is Proposition \ref{prop:fslvcs}. It is easy to see that $\mathrm{FSL} \leqW \C_\mathbb{N}$, as we can just guess a potential leaf and the time it will be enumerated into the tree. By Proposition \ref{prop:cbelimination} then $\VC_{S_0} \leqW \mathrm{FSL} \leqW \C_\mathbb{N}$ would imply that $\VC_{S_0}$ is computable, and thus also that $\mathrm{FSL}$ is computable by Proposition \ref{prop:fslvcs}. But that contradicts Proposition \ref{prop:fslnc}.
\end{proof}
\end{corollary}

\subsection{Open questions}
\label{subsec:qb:questions}
We have presented our results on overt choice for countably-based non-quasi-Polish spaces not with the intention of concluding their investigation, but in the hope to spark further interest. We do not even dare to list a comprehensive list of open questions that we deem worthy of future work, but instead only list some prototypical questions.

If we consider upper bounds for overt choice amongst the usual Weihrauch degrees used for calibration in the literature, we see a huge gap between our negative result ruling out $\C_\mathbb{N}$ (Proposition \ref{prop:cbelimination}) and the positive answer providing $\C_\Baire$ as upper bound for a large class of spaces (Corollary \ref{corr:cbaireupper}). We thus ask whether this can be tightened. On the upper end of that gap, an initial question would be whether $\UC_\Baire$ might suffice\footnote{While a number of intermediate (between $\UC_\Baire$ and $\C_\Baire$) principles were studied in \cite{pauly-kihara4}, $\UC_\Baire$ still seems like a reasonable step down from $\C_\Baire$.}.
\begin{question}
Is there an effectively analytic effectively countably-based space $\mathbf{X}$ with $\VC_\mathbf{X} \nleqW \UC_\Baire$?
\end{question}

On the lower end of the gap, comparing $\VC_\mathbb{Q}$ and $\VC_{S_0}$ with degrees such $\lim$ and $\Sort$. Between $\C_\mathbb{R}$ not being an upper bound (Corollary \ref{corr:cvqncr}) and $\Pi^0_2\C_\mathbb{N}$ serving as such (Corollary \ref{corr:vcqpi02}), these would seem to be the next suitable candidates:

\begin{question}
Are $\VC_\mathbb{Q}$ and/or $\VC_{S_0}$ reducible to $\lim$, or even to $\Sort$?
\end{question}

\begin{question}
How are $\VC_\mathbb{Q}$ and $\VC_{S_0}$ related?
\end{question}

It seems very desirable to study overt choice for a broader range of countably-based non-quasi-Polish spaces than just two examples (however well the choice of these is motivated). Another simple and natural example would be the space $\mathbb{N}_{\mathrm{cof}}$ of integers with the cofinite topology (this is essentially the subspace $\{\{n\} \in \mathcal{A}(\mathbb{N}) \mid n \in \mathbb{N}\} \subseteq \mathcal{A}(\mathbb{N})$. This space is the typical example of a $T_1$ non-$T_2$-space. As such, we know from Theorem \ref{theo:vcimpliesqp} below that $\VC_{\mathbb{N}_{\mathrm{cof}}}$ is discontinuous -- but without a concrete proof giving us a meaningful lower bound in the Weihrauch lattice.

\begin{question}
What else can we say about $\VC_{\mathbb{N}_{\mathrm{cof}}}$?
\end{question} 

\section{Overt choice for CoPolish spaces}
\label{sec:copolish}
\subsection{Background on coPolish spaces}
In general, non-countably based spaces are often very difficult to understand (see e.g.~\cite{hoyrup8}). A nice class of not-necessarily countably-based topological spaces is formed by the class of CoPolish spaces.
They play a role in Type-2-Complexity Theory \cite{schroder6}
by allowing simple complexity. Concrete examples of CoPolish spaces relevant for analysis include the space of polynomials over the reals, the space of analytic functions and the space of compactly-supported continuous real functions.

\begin{definition}[\cite{schroder6}]
 A \emph{CoPolish space} $\XX$ is the direct limit
 of an increasing sequence of compact metrisable subspaces $\XX_k$.
\end{definition}

Any CoPolish space is a Hausdorff normal qcb-space.
We present a characterization of  Copolish spaces.

\begin{proposition}[\cite{schroder6}] \label{p:characterization:CoPolish}
 Let $\XX$ be a Hausdorff qcb-space. Then the following are equivalent:
\begin{enumerate}
 \item
  $\XX$ is a CoPolish space.
 \item
  The space $\mathcal{O}(\XX)$ of open subsets of $\XX$ equipped with the Scott-topology
  is a quasi-Polish space and $\XX$ is regular.
 \item
  $\XX$ has an admissible representation with a locally compact domain.
 \item
  $\XX$ has a countable pseudobase consisting of compact subsets.
\end{enumerate}
\end{proposition}

In the realm of countably-based Hausdorff spaces, Copolishness is just
 local compactness.

 \begin{lemma}\label{l:Copolish:vs:localcompact}
  A countably-based space is CoPolish if, and only if, if it is locally
 compact.
 \end{lemma}

 \begin{proof}
  Let $\XX$ be a locally compact Hausdorff space with countable basis
 $\mathcal{B}$.
  Then the countable subfamily $\mathcal{B}'$ of basic open sets whose
 closure is compact forms a basis as well.
  The family $\mathcal{K}$ of closures of sets in $\mathcal{B}'$ is then
 a countable pseudobase for $\XX$ consisting of compact sets.
  \\
  Conversely, let $\XX$ be the direct limit of a increasing sequence of
 compact metrisable subspaces $\XX_m$.
  Let $x$ be a point in $\XX$ with countable neighbourhood basis $\{B_i
 \,|\, i \in \NN\}$.
  Assume for contradiction $\bigcap_{i=0}^n B_i \nsubseteq \XX_n$ for
 all $n \in \NN$.
  Then for every $n$ there exists some $y_n \in \bigcap_{i=0}^n B_i
 \setminus \XX_n$.
  Clearly $(y_n)_n$ converges to $x$, so there is some $m$ such that
 $\{x,y_n \,|\, n \in \NN \} \subseteq \XX_m$,
  a contradiction.
  \\
  We conclude that $\XX_m$ is a compact neighbourhood of $x$.
  By Hausdorffness this implies that $\XX$ is locally compact.
 \end{proof}

 CoPolish spaces can be separated into three classes,
 the countably-based ones, the non-countably-based Fr{\'e}chet-Urysohn
 ones and the non-Fr{\'e}chet-Urysohn ones.

\subsection{Fr{\'e}chet-Urysohn spaces}

A topological space $X$ is called \emph{Fr{\'e}chet-Urysohn},
if the closure of any subset $M$ is equal to the set of all limits of sequences in $M$.
Any countably-based space and any metrisable space is a Fr{\'e}chet-Urysohn space.
We present an example of a Fr{\'e}chet-Urysohn CoPolish space $\Tmin$ that does not have a countable base.
In Lemma~\ref{l:FUCoPol:props} we will see that $\Tmin$ is a minimal such space.

\begin{example}
 The underlying set of $\Tmin$ is $\NN^2 \cup \{\infty\}$.
 A basis of the topology is given by the sets
 \[
  \{(a,b)\} \quad\text{and}\quad
  U_\ell:=\{\infty \} \cup \big\{(a,b)\,\big|\, b\geq \ell(a)\big\}
 \]
 for all $(a,b) \in \NN^2$ and $\ell \in \Baire$.
 Clearly, $\Tmin$ is the direct limit of the compact subspaces $\XX_m$
 that have $\{\infty\} \cup \{(a,b)\,|\, a \leq m \}$ as their respective  underlying sets.
 So $\Tmin$ is CoPolish.
 It is Fr{\'e}chet-Urysohn, because it is sequential and has only one point that does not form an open singleton.
 A computably admissible representation $\delta_{\Tmin}$ for $\Tmin$ has
 $\big\{m0^\omega, m0^b(a+1)0^\omega \,\big|\, a,b,m \in \NN,\, a \leq m \big\}$
 as its locally compact domain.
 It maps $m0^\omega$ to $\infty$ and $m0^b(a+1)0^\omega$ to $(a,b)$.
\end{example}

Overt choice on $\Tmin$ is not computable, because all-or-co-unique-choice on the natural numbers, denoted by $\mathrm{ACC}_\NN$ in
\cite{pauly-handbook},
is Weihrauch-reducible to $\VC_\Tmin$.

\begin{proposition}
\label{prop:accNtmin}
 $\mathrm{ACC}_\NN \leqsW \VC_\Tmin$.
\end{proposition}

\begin{proof}
 $\mathrm{ACC}_\NN$ is the problem of finding an element
 in a given set $A$ in the family $\{ \NN, \NN \setminus \{i\} \,|\, i \in \NN \} \subseteq \mathcal{A}(\mathbb{N})$.
 A computably admissible representation $\psi$ of this family is given by
 \[
  \psi(0^\omega):= \NN \quad\text{and}\quad
  \psi\big(0^j(i+1)0^\omega\big):= \NN \setminus \{i\} \,.
 \]
 We define the preprocessor $K$ to map $0^\omega$ to the set $\{\infty\}$
 and $0^j(i+1)0^\omega$ to the closed set $\big\{(i+1,j),(0,i)\big\}$.
 Let the postprocessor $H\colon \dom(\Tmin) \to \NN$ be defined by
 \[
  H(m0^\omega):=m \quad\text{and}\quad
  H\big(m0^b(a+1)0^\omega\big):= \left\{
  \begin{array}{cl}
    m+1 & \text{if $b=m$}  \\ 
    m   & \text{otherwise.}
  \end{array}\right.
 \]
 Let $p \in \dom(\psi)$ and $A:=\psi(p)$.
 If a realizer $G$ of overt choice applied to $K(p)$ returns $m0^\omega$, then the input set $A$ is $\NN$ so that $m$ is a legitimate result.
 If $G$ returns $m0^b(a+1)0^\omega$, then $A$ is
 either $\NN \setminus \{a-1\}$ or $\NN \setminus \{b\}$.
 As $a \leq m$, we have $H\big(m0^b(a+1)0^\omega \big) \in A$ as required. It is easy to see that $K$ and $H$ are both computable.
\end{proof}

We list a few properties of Fr{\'e}chet-Urysohn CoPolish spaces that will be instrumental to understand the complexity of their overt choice principles:

\begin{lemma}\label{l:FUCoPol:props}
 Let $\XX$ be a Fr{\'e}chet-Urysohn CoPolish space.
\begin{enumerate}
  \item
   The subspace $\XX_\omega$ of the points in $\XX$ that have a countable neighbourhood base
   is open.
  \item
   The complement $\XX_{\mathrm{nc}}:= \XX \setminus \XX_\omega$ forms a closed and discrete subspace of $\XX$.
  \item
   If $\XX_{\mathrm{nc}} \neq \emptyset$, then $\Tmin$ embeds into $\XX$ as a closed subspace.
\end{enumerate}
\end{lemma}
\begin{proof}
  Let $(\XX_m)_m$ an increasing sequence of compact metrisable subspaces
 such that $\XX$ is the direct limit of $(\XX_m)_m$.
 \begin{enumerate}
  \item \label{en:XXomega:open}
   In the proof of Lemma~\ref{l:Copolish:vs:localcompact}  we have seen
 that for any point $x$ with a countable neighbourhood base there is some
 $m$ such that $x$ is in the interior $\mathrm{int}(\XX_m)$ of some
 $\XX_m$.
   Since $\XX_m$ has a countable base, any point in the interior of
 $\XX_m$ has a countable neighbourhood base in $\XX$,
   namely the one in the subspace $\mathrm{int}(\XX_m)$.
   Hence $\XX_\omega= \bigcup_m \mathrm{int}(\XX_m)$ is open.
  \item \label{en:XXnc:discrete}
   For any point $x \in \XX_{\mathrm{nc}}$ and any $m \in \NN$
   there is a sequence $(y_i)_i$ outside $\XX_m$ which converges to $x$,
   as otherwise $x$ were in the interior of $\XX_m$ due to the
 Fr{\'e}chet-Urysohn property,
   which would imply $x \in \XX_\omega$ by the discussion in item
 \eqref{en:XXomega:open}.
   \\
   Assume that there exists an injective sequence $(x_n)_n$ in
 $\XX_{\mathrm{nc}}$
   that converges to some point $x_\infty \in \XX_{\mathrm{nc}}$.
   W.l.o.g.\ $x_\infty \notin \{ x_n \,|\, n \in \NN\}$.
   Set $m_{-1}:=\min\big\{ i \in \NN \,\big|\, \{x_\infty,x_n \,|\, n
 \in \NN \} \subseteq \XX_i \big\}$.
   By the above observation we can construct an increasing sequence
 $(m_a)_a$ of natural numbers strictly above $m_{-1}$
   and a double sequence $(y_{a,b})_{a,b}$ such that $(y_{a,b})_b$
 converges to $x_a$
   and $y_{a,b} \in \XX_{m_a} \setminus \XX_{m_{a-1}}$ for every $a,b$.
   Obviously, $x_\infty$ is in the closure of $\{ y_{a,b} \,|\, a,b \in
 \NN\}$.
   So there are functions $s,t\colon \NN \to \NN$ such that
 $(y_{s(i),t(i)})_i$ converges to $x_\infty$ by the Fr{\'e}chet-Urysohn
 property.
   Then $s$ is bounded, because any convergent sequence in $\XX$ is
 contained in some subspace $\XX_m$.
   But then there is a subsequence of $(y_{s(i),t(i)})_i$ converging
 either to some $x_n \neq x_\infty$ or to some $y_{a,b} \neq x_\infty$.
 This contradicts the Hausdorff property.
   \\
   We conclude that in $\XX_{\mathrm{nc}}$ all converging sequences are
 eventually constant.
   Since $\XX_{\mathrm{nc}}$ is sequential by being a closed subspace of
 a qcb-space and Hausdorff,
   $\XX_{\mathrm{nc}}$ is discrete.
   Discrete qcb-spaces are at most countable because of the existence of
 a countable pseudobase.
  \item
   Choose some point $x \in \XX_{\mathrm{nc}}$.
   Set $m_{-1}:=\min\{ i \,|\, x \in \XX_i\}$.
   In a similar way as in the proof of \eqref{en:XXnc:discrete},
   we construct an increasing sequence $(m_a)_a$ of natural numbers
 strictly above $m_{-1}$
   and a double sequence $(y_{a,b})_{a,b}$ such that $(y_{a,b})_b$
 converges (now) to $x$
   and $y_{a,b} \in \XX_{m_a} \setminus \XX_{m_{a-1}}$ for every $a,b$.
   Since $\XX$ is Hausdorff and $y_{a,b} \neq x$,
   for all $a$ the sequence $(y_{a,b})_b$ contains an injective
 subsequence $(z_{a,j})_j$.
   We define $e\colon \Tmin \to \XX$ by $e(\infty):=x$ and
 $e(a,b):=z_{a,b}$.
   By construction $e$ is injective and continuous.
   Moreover if $(t_n)_n$ is an injective sequence in $\Tmin$ such that
 $(e(t_n))_n$ converges to some point $y$ in $\XX$,
   then $y=x$ by the Hausdorffness of $\XX$ and again by the fact that
 any convergent sequence is contained in some $\XX_m$.
   So the image of $e$ is closed and $e$ reflects converging sequences
   (meaning that $(t_n)_n$ converges to $t_\infty$, whenever
 $(e(t_n))_n$ converges to $e(t_\infty)$).
   Therefore $\Tmin$ embeds topologically into $\XX$ as a closed
 subspace.
 \end{enumerate}
 \end{proof}

\begin{theorem}
  Let $\XX$ be a Fr{\'e}chet-Urysohn CoPolish space.
  Then overt choice on $\XX$ is continuous if, and only if, $\XX$ is countably-based.
\end{theorem}

\begin{proof}
 If $\XX$ has a countable base, then $\XX$ is locally compact and therefore a Polish space.
 Hence overt choice on $\XX$ is continuous (see \cite{presser}). 
 \\
 If $\XX$ is not countably-based, then $\XX$ is not first-countable by \cite[Proposition 3.3.1]{schroder5}, thus $\XX_{\mathrm{nc}} \neq \emptyset$.
 Therefore $\Tmin$ embeds topologically into $\XX$ as a closed subspace by Lemma~\ref{l:FUCoPol:props}.
 Since $\VC_\Tmin$ is discontinuous by Proposition \ref{prop:accNtmin}, $\VC_\XX$ is discontinuous as well.
\end{proof}

Overt choice on Fr{\'e}chet-Urysohn CoPolish spaces turns out to have $\lpo$ as an upper bound in the topological Weihrauch lattice.
Remember that $\lpo\colon \Baire \to \{0,1\}$ is defined by
$\lpo(r)=1 :\Longleftrightarrow \exists k \in \NN. r(k)=0$.

\begin{theorem}\label{th:FUCoPol:lpo}
 Let $\XX$ be a Fr{\'e}chet-Urysohn CoPolish space.
 Then $\VC_\XX \leqWcont \lpo$.
\end{theorem}

\begin{proof}
 Given a positive name $p$ of a non-empty closed set $A$, we first use $\lpo$ to decide whether or not $A$ intersects the open set $\XX_\omega$.
\begin{enumerate}
 \item
  If it does, we proceed as follows.
  By being an open subspace of a CoPolish space, $\XX_\omega$ is CoPolish as well,
  because those elements of a countable compact pseudobase that are contained in $\XX_\omega$ form a countable compact pseudobase for $\XX_\omega$.
  Since $\XX_\omega$ is first-countable, it has a countable base by \cite[Proposition 3.3.1]{schroder5}.
  Therefore $\XX_\omega$ is locally compact and Polish.
  Since $\mathcal{O}(\XX_\omega)$ is a retract of $\mathcal{O}(\XX)$,
  we can continuously convert the given name of $A$ into a positive name of the closed subset $A \cap \XX_\omega$ in the space $\XX_\omega$.
  Now we can apply the continuous version of the algorithm from \cite{presser} to obtain an element of $A \cap \XX_\omega$.
 \item
  Now we consider the case $A \subseteq \XX_{\mathrm{nc}}$.
  Since $\XX_{\mathrm{nc}}$ is a discrete qcb-space, it is countable.
  So there are elements $z_i$ with $\{ z_i \,|\, i \in \NN\}=\XX_{\mathrm{nc}}$.
  By Lemma~\ref{l:FUCoPol:props}
  the sets $W_i:=\{z_i\} \cup \XX_\omega$ are open.
  By dovetailing we systematically search for a set $W_i$ that intersects $A$.
  Once we have found one, we output the corresponding element $z_i$.
\end{enumerate}
\end{proof}

\begin{remark}
 If we require that the subspace $\XX_\omega$ of $\XX$ is computable equivalent to a computable Polish space, the set $\XX_\omega$ is computably open in $\XX$ and the elements of $\XX_{\mathrm{nc}}$ form a computable sequence,
 then we have $\VC_\XX \leqW \lpo$.
\end{remark}

We proceed to show that the reduction in Theorem \ref{th:FUCoPol:lpo} is strict by revealing the weakness of $\VC_\mathbf{X}$ for Fr\'echet-Urysohn spaces $\mathbf{X}$. Again we use a technical lemma:

\begin{lemma}\label{l:charact:convrel:FUspace}
  Let $\XX$ be a admissibly represented Fr{\'e}chet-Urysohn space,
  and let $(A_n)_{n \leq \infty}$ be a sequence of non-empty closed
 subsets.
  Then $(A_n)_n$ converges to $A_\infty$ in $\mathcal{V}(\XX)$ if, and
 only if,
  for any $x_\infty \in A_\infty$ and any strictly increasing function
 $\varphi\colon\NN \to \NN$
  there is a sequence $(x_n)_n$ converging to $x_\infty$ and a strictly
 increasing function $\xi\colon\NN \to \NN$
  with $x_n \in A_{\varphi\xi(n)}$ for all $n \in \NN$.
 \end{lemma}

 \begin{proof}
  The backward direction is obvious.
  For the forward direction, let $(A_n)_n$ be a sequence of non-empty
 closed sets converging in $\mathcal{V}(\XX)$ to $A_\infty$ and let
 $x_\infty \in A_\infty$.
  It suffices to consider $\varphi= \mathrm{id}_\NN$.
  If $x_\infty$ is contained in $A_n$ for infinitely many $n$, then we
 simply choose $x_n=x_\infty$
  and $\xi$ as the strictly increasing function with
 $\mathrm{range}(\xi)=\{n \in \NN \,|\, x_\infty \in A_n\}$.\
  \\
  Otherwise there is some $n_0$ with $x_\infty \notin \bigcup_{i \geq
 n_0} A_i$.
  Since $(A_n)_{n \geq n_0}$ converges to $A_\infty$ with respect to the
 lower fell topology,
  $x_\infty$ is in the closure of $\bigcup_{i \geq n_0} A_i$.
  By the Fr{\'e}chet-Urysohn property, there is a sequence in $(y_m)_m
 \in \bigcup_{i \geq n_0} A_i$
  converging to $x_\infty$.
  As the closed set $\bigcup_{i=n_0}^n A_i$ does not contain $x_\infty$,
 it contains $y_n$ for finitely many $n$'s.
  So we have a strictly increasing sequence $(m_n)_n$ such that $y_m
 \notin \bigcup_{i=n_0}^n A_i$ for all $m \geq m_n$.
  We inductively define $x_0:=y_0$, $\xi(0):=\min\{i \geq n_0 \,|\, x_0
 \in A_i\}$,
  $x_{k+1}:=y_{m_{\xi(k)}}$ and $\xi(k+1):=\min\{i \geq n_0 \,|\,
 x_{k+1} \in A_i\}$.
  Clearly $\xi(k+1)>\xi(k)$ and thus $m_{\xi(k+1)} > m_{\xi(k)}$.
  So $(x_k)_k$ converges to $x_\infty$.
 \end{proof}

Recall that $\mathrm{ACC}_m$ is the problem of finding an element in a given
 closed set $A$ in the family $\{ M, M \setminus \{i\} \,|\, i \in \NN
 \}$, where $M=\{0,\dotsc,m-1\}$. Note that $\mathrm{ACC}_2 \equivW \llpo$.
  A computably admissible representation $\psi$ of this family is given
 by
  \[
   \psi(0^\omega):= M \quad\text{and}\quad
 \psi\big(0^j(i+1)0^\omega\big):= M \setminus \{i\} \,.
  \]

 \begin{theorem}
 \label{theo:FUweak}
  Let $\XX$ be a admissibly represented Fr{\'e}chet-Urysohn space and
 let $m\geq 2$.
  Then $\mathrm{ACC}_m \nleqW \VC_\XX$.
 \end{theorem}

 \begin{proof}

  Assume, there were continuous functions $K\colon \dom(\psi) \to
 \mathcal{V}(\XX)$ and $H\colon \dom(\psi) \times \dom(\delta_\XX) \to
 \NN$ witnessing $\mathrm{ACC}_m \leqW \VC_\XX$.
  \\
  We choose some $x_\infty \in K(0^\omega)$.
  Let $a \in M$.
  Since $K\big(0^n(a+1)0^\omega \big)$ converges to $K(0^\omega)$,
  by the above lemma there is a strictly increasing function
 $\xi_a\colon \NN \to \NN$ and a sequence $(y_{a,n})_n$
  converging to $x_\infty$ such that $y_{a,n} \in K\big(
 0^{\xi_a(n)}(a+1)0^\omega \big)$.
  The sequence $(x_n)_n:=(y_{n \,\mathrm{mod}\, m,n \,\mathrm{div}\,
 m})_n$ converges to $x_\infty$ as well,
  because $M$ is finite.
  Since $\delta_\XX$ is admissible,
  there is a sequence $(s_n)_n$ converging to some name $s_\infty$ of
 $x_\infty$ such that $\delta_\XX(s_n)=x_n$.
  Now we consider $b:=H(0^\omega,s_\infty)$.
  For almost all $n$ we have
   $H\big( 0^{\xi_b(n)}(b+1)0^\omega,s_{mn+b} \big)=b$
   and $\delta_\XX(s_{mn+b})=y_{b,n} \in K\big(
 0^{\xi_b(n)}(b+1)0^\omega  \big)$,
  contradicting $b \notin \psi\big( 0^{\xi_b(n)}(b+1)0^\omega \big)=M
 \setminus \{b\}$.
 \end{proof}

\subsection{Non-Fr{\'e}chet-Urysohn spaces}

Now we turn our attention to non-Fr{\'e}chet-Urysohn $T_1$-spaces.
First we show that overt choice on them is above $\lpo$ in the continuous Weihrauch lattice.

\begin{theorem}\label{th:lpo:VCNonFU}
 Let $\YY$ be an admissibly represented space such that its topology is $T_1$, but not Fr{\'e}chet-Urysohn.
 Then $\lpo \leqWcont \VC_\YY$.
\end{theorem}

\begin{proof}
 We choose a subset $M$ such that on the one hand there is some $y$ in the closure of $M$,
 but on the other hand $y$ is not the limit of any sequence in $M$.
 By \cite[Proposition 3.3.1]{schroder5}, $M$ equipped with the subsequential topology contains a dense sequence $(z_i)_i$.
 For $n \in \NN$ we define the closed set $A_n$ by $\{ z_i \,|\, i \leq n \}$.
 Then $(A_n)_n$ converges to $A_\infty:=\{y\}$ in the lower Fell topology.
 Since the standard positive representation of closed is admissible w.r.t.\ the lower Fell topology by \cite[Proposition 4.4.5]{schroder5},
 there is a sequence $(p_n)_n$ of names for the $A_n$'s converging to some name $p_\infty$ of $A_\infty$.
 We define a continuous function $K\colon \Baire \to \Baire$ by
 \[
  K(r):=\left\{
  \begin{array}{ll}
    p_\infty & \text{if $r$ does not contain $0$}
    \\
    p_m & \text{if $m:=\min\big\{k \in \NN \,\big|\, r(k)=0\big\}$ exists.}
  \end{array}\right.
 \]
 The singleton $\{y\}$ is sequentially open
 in the subspace of $\YY$ with underlying set $\{y, z_i \,|\, i \in \NN\}$,
 because no sequence in $\{ z_i \,|\, i \in \NN\}$ converges to $y$.
 By the $T_1$-property, $\{y\}$ is even clopen.
 So there is continuous function $H:\subseteq \Baire \to \{0,1\}$ such that
 \[
   \delta_\YY(r)=y \implies H(r)=1
   \quad\text{and}\quad
   \delta_\YY(r)\in \{z_i\,|\, i \in \NN\} \implies H(r)=0 \,.
 \]
 Clearly, for any realizer $G$ of overt choice on $\YY$
 the function $HGK$ is a realizer for $\lpo$.
 Hence $\lpo \leqsWcont \VC_\YY$.
\end{proof}

We present an example of a CoPolish non-Fr{\'e}chet-Urysohn space $\Smin$ for which overt choice is Weihrauch equivalent to $\lpo$.

\begin{example}
 We choose
 $\{ (\infty,\infty) \} \cup ( \NN \times \{\infty\} ) \cup \NN^2$
 as the underlying set of $\Smin$.
 The topology of $\Smin$ is induced by the basis consisting of the sets
 \begin{itemize}
   \item $\{(a,b)\}$,
   \item $\{(a,\infty)\} \cup \{ (a,j) \,|\,  j \geq b \}$,
   \item
     $\{(\infty,\infty)\} \cup \{(i,\infty) \,|\, i\geq a\} \cup \{ (i,j) \,|\, i\geq a, j \geq \ell(i) \}$
 \end{itemize}
 for all $a,b \in \NN$, $\ell \in \Baire$.
 The space is not Fr{\'e}chet-Urysohn, because $(\infty,\infty)$ belongs to the closure of $\NN^2$, but fails to be a limit of any sequence in $\NN^2$.
 An admissible representation $\delta_\Smin$ of $\Smin$ has
 $\big\{i0^\omega, i0^j1^\omega \,\big|\, i,j \in \NN \big\}$ as its locally compact domain
 and is defined by
 \[
    \delta_\Smin(00^\omega):=(\infty,\infty),\,
    \delta_\Smin(00^a1^\omega):=\delta_\Smin\big((a+1)0^\omega\big):=(a,\infty),\,
    \delta_\Smin\big((a+1)0^b1^\omega\big):=(a,b)
  \]
 for all $a,b \in \NN$.
 A proof of admissibility can be found in \cite[Example 2.3.15]{schroder5}.
 The space $\Smin$ is the direct limit of its compact subspaces
 $\XX_m=\{ (\infty,\infty), (a,\infty), (i,b) \,|\, a,b \in \NN, i \leq m\}$.
 Hence $\Smin$ is CoPolish.
\end{example}

We remark that any CoPolish space that is not Fr{\'e}chet-Urysohn contains a copy of $\Smin$ as a closed subspace.
Now we show that overt choice on $\Smin$ is Weihrauch-equivalent to $\lpo$.

\begin{theorem}
 $\VC_\Smin \equivW \lpo$.
\end{theorem}

\begin{proof}
 For the direction $\lpo \leqW \VC_\Smin$ we effectivize the proof of Theorem~\ref{th:lpo:VCNonFU}.
 We choose a computable pairing function $\langle \cdot,\cdot\rangle$
 and define $A_n:= \big\{ (a,b) \,\big|\, \langle a,b\rangle \leq n \big\}$.
 Clearly, the sequence $(A_n)_n$ converges effectively to
 the singleton $\{(\infty,\infty)\}$ in the space of closed sets with the positive representation.
 So we have a computable sequence $(p_n)_n$ of names of the $A_n$'s
 which converges computably to a name $p_\infty$ of $\{(\infty,\infty)\}$.
 Therefore the preprocessor $K$ defined like in the proof of Theorem~\ref{th:lpo:VCNonFU}
 is computable.
 The computable postprocessor $H\colon \dom(\delta_{\Smin}) \to \{0,1\}$
 can be defined by $H(s)=1 :\Longleftrightarrow s(0)=0$.
 Clearly, $K$ and $H$ witness $\lpo \leqsW \VC(\Smin)$.
\medskip\\
 To show $\VC_\Smin \leqW \lpo$, we first employ $\lpo$ to decide whether or not the computably open set $\XX_\omega:=\Smin \setminus \{(\infty,\infty)\}$ intersects the given non-empty closed set $A$.
 If not, then we output the element $(\infty,\infty)$.
 In the positive case we employ the fact that
 the space $\XX_\omega$ forms a computable Polish space.
 Moreover, we can compute a name of $A \cap \XX_\omega$ in the positive representation for the closed subsets of $\XX_\omega$.
 Hence the algorithm from \cite{presser} computes for us an element
 of $A \cap \XX_\omega \neq \emptyset$.
\end{proof}

\subsection{Upper bound for overt choice on coPolish spaces}

Let the space $\Sigma_{\Pi^1_1}$ have the two elements $\top$ and $\bot$, with $p \in \Baire$ being a name for $\top$ iff $p$ codes a well-founded tree, and a name for $\bot$ iff it codes an ill-founded tree. The map $\id : \Sigma_{\Pi^1_1} \to \Sigma$ essentially lets us treat a single $\Pi^1_1$-set as an open set. Alternatively, we can view $\id : \Sigma_{\Pi^1_1} \to \Sigma$ as testing whether some $A \in \mathcal{A}(\Baire)$ is empty. We find that $\left (\id : \Sigma_{\Pi^1_1} \to \Sigma \right ) \leqW \mathrm{TC}_\Baire$ for the principle $\mathrm{TC}_\Baire$ introduced and studied in \cite{pauly-kihara4}, and that $\Pi^1_1\mathrm{CA} \equivW \lim \star \widehat{\left (\id : \Sigma_{\Pi^1_1} \to \Sigma\right)}$. In particular, it holds that $\left (\id : \Sigma_{\Pi^1_1} \to \Sigma \right ) \nleqW \C_\Baire$.

\begin{theorem}
\label{theo:copolishupper}
Let $\mathbf{X}$ be coPolish. Then $\VC_\mathbf{X} \leqWcont \widehat{\left (\id : \Sigma_{\Pi^1_1} \to \Sigma\right)}$.
\begin{proof}
We use the characterization of $\mathbf{X}$ as being a direct limit of a sequence of compact Polish spaces $\mathbf{K}_0 \embeds \mathbf{K}_1 \embeds \ldots$. Given some basic open set $\sigma$ of $\mathbf{K}_\ell$ and $f \in \Baire$, we inductively define $U^0(\sigma, f) = \sigma$ and $U^{n+1}(\sigma,f) = \{x \in \mathbf{K}_{\ell+n+1} \mid d(x,U^n(\sigma,f)) < 2^{-f(n)}\}$. Then let $U(\sigma,f) \subseteq \mathbf{X}$ be the corresponding direct limit. Note that $U(\sigma,f)$ is open in $\mathbf{X}$, and that the sets of the form $U(\sigma,f)$ form a basis of $\mathbf{X}$.

Now given some $A \in \mathcal{V}(\mathbf{X})$, we find that $\sigma \cap A \neq \emptyset$ iff $\forall f \in \Baire \ U(\sigma,f) \cap A \neq \emptyset$. Since $U(\sigma,f) \cap A \neq \emptyset$ is an open property, we will recognize that it holds true for some $f$ based on some finite prefix of $f$. From $A$ and $\sigma$ we can thus construct a tree $T$ such that the paths through $T$ are exactly those $f$ with $A \cap U(\sigma,f) = \emptyset$. Thus, we can use $\widehat{\left (\id : \Sigma_{\Pi^1_1} \to \Sigma\right)}$ to obtain a list of all basic open sets in any $\mathbf{K}_i$ intersecting $A$. Once we have identified a $\mathbf{K}_i$ with $\mathbf{K}_i \cap A \neq \emptyset$, this lets us obtain $\mathbf{K}_i \cap A \in \mathcal{V}(\mathbf{K}_i)$, from which we can compute a point $x \in \mathbf{K}_i \cap A$ since overt choice on Polish spaces is computable. We then translate $x \in \mathbf{K}_i$ into $x \in \mathbf{X}$.
\end{proof}
\end{theorem}

\subsection{Summary}

We obtain the following corollary to Theorems~\ref{th:FUCoPol:lpo}, \ref{theo:FUweak}, \ref{th:lpo:VCNonFU} and \ref{theo:copolishupper}. It shows that the topological Weihrauch degree of overt choice for a CoPolish space characterizes whether or not the space is countably-based, and whether or not the space has the Fr\'echet-Urysohn property:

\begin{corollary}
\label{corr:copolishsummary}
For a CoPolish space $\mathbf{X}$ exactly one of the following cases holds:
\begin{enumerate}
\item $\XX$ is Polish and $\VC_\XX$ is continuous.
\item $\XX$ is not countably-based, Fr\'echet-Urysohn, and $\mathrm{ACC}_\mathbb{N} \leqWcont \VC_\XX \leWcont \lpo$.
\item $\XX$ is not Fr\'echet-Urysohn, and $\lpo \leqWcont \VC_\XX \leqWcont \widehat{\left (\id : \Sigma_{\Pi^1_1} \to \Sigma\right)}$.
\end{enumerate}
\end{corollary}

How much the Fr\'echet-Urysohn property fails for a sequential space can be characterized by the ordinal invariant $\sigma$ defined in \cite{arhangelskii}. $\sigma$ specifies how many times you need to iterate sequential closures to get the closure of an arbitrary subset. We wonder whether a more precise classification of overt choice for CoPolish spaces might be achievable depending on $\sigma$. Note that the interval of the Weihrauch lattice we know $\VC_\XX$ to fall into in this case contains the Baire hierarchy. 



\bibliographystyle{eptcs}
\bibliography{references}
\end{document}